\documentclass[12pt]{amsart}
\usepackage{graphicx,fullpage}
\usepackage[pdftex,colorlinks,citecolor=blue,urlcolor=blue]{hyperref}
\usepackage{amssymb}
\usepackage{graphicx,color}
\usepackage{tikz}
\usepackage{amsmath}
\usepackage{comment}
\usepackage{MnSymbol}
\usepackage{cleveref}
\usepackage{bbold}
\usepackage{bbm}
\usepackage{tabularx}

\numberwithin{equation}{section}
\newtheorem{theorem}{Theorem}
\newtheorem{proposition}[theorem]{Proposition}

\newtheorem{corollary}[theorem]{Corollary}
\newtheorem{conjecture}[theorem]{Conjecture}
\newtheorem{lemma}[theorem]{Lemma}

\newtheorem{problem}[theorem]{Problem}

\newtheorem{definition}[theorem]{Definition}
\newtheorem{example}[theorem]{Example}
\newtheorem{remark}[theorem]{Remark}

\newcommand{\A}{\mathbb{A}}
\newcommand{\R}{\mathbb{R}}
\newcommand{\C}{\mathbb{C}}
\newcommand{\N}{\mathbb{N}}
\newcommand{\PP}{\mathbb{P}}
\newcommand{\RP}{\mathbb{P}_{\mathbb{R}}}

\newcommand{\cV}{\mathcal{V}}
\newcommand{\cO}{\mathcal{O}}
\newcommand{\cE}{\mathcal{E}}
\newcommand{\Vector}[1]{\mathbf{#1}}
\newcommand{\x}{\mathbf{x}}
\newcommand{\X}{\mathbf{X}}

\newcommand{\z}{\mathbf{z}}

\newcommand{\bV}{\mathbf{V}}
\newcommand{\dd}{\mathrm{\normalfont{d}}}

\newcommand{\New}{\mathrm{\normalfont{New}}}

\newcommand{\rr}{{\mathbb R}}

\newcommand{\al}{\alpha}
\newcommand{\be}{\beta}

\newcommand{\ep}{\varepsilon}

\title{
  On odd powers of nonnegative polynomials \\ that are not sums of squares}

\author[G. Blekherman]{Grigoriy Blekherman}
\address{School of Mathematics, Georgia Institute of Technology, 686 Cherry Street Atlanta, GA 30332, USA}
\email{greg@math.gatech.edu}

\author[K. Kozhasov]{Khazhgali Kozhasov}
\address{Laboratoire J.-A. Dieudonn\'e \\ Universit\'e C\^ote d'Azur \\ Parc Valrose, 06108 Nice, France}
\email{khazhgali.kozhasov@univ-cotedazur.fr}

\author[B. Reznick]{Bruce Reznick}
\address{Department of Mathematics, University of Illinois at Urbana-Champaign, Urbana, IL 61801}
\email{reznick@illinois.edu}

\date{}

\begin{document}

\begin{abstract}
  We initiate a systematic study of nonnegative polynomials $P$ such that $P^k$ is not a sum of squares for any odd $k\geq 1$, calling such $P$ \emph{stubborn}. We develop a new invariant of a real isolated zero of a nonnegative polynomial in the plane, that we call \emph{the SOS-invariant}, and relate it to the well-known delta invariant of a plane curve singularity. Using the SOS-invariant we show that any polynomial that spans an extreme ray of the convex cone of nonnegative ternary forms of degree 6 is stubborn. We also show how to use the SOS-invariant to prove stubbornness of ternary forms in higher degree. Furthermore, we prove that in a given degree and number of variables, nonnegative polynomials that are not stubborn form a convex cone, whose interior consists of all strictly positive polynomials.

\end{abstract}

\maketitle

\section{Introduction}





The relationship between nonnegative polynomials and sums of squares is a fundamental problem in real algebraic geometry. Much is now known about constructions and existence of nonnegative polynomials that are not sums of squares (of polynomials) \cite{CL2, Reznick1989FormsDF, Reznick2007OnHC, Blekherman2010NonnegativePA, Blekherman2015, Brugall2018RealAC}. Decomposing $P^k$ as a sum of squares, where $k$ is an odd integer, also provides a certificate of nonnegativity of $P$, and it is reasonable to ask whether this works for all nonnegative polynomials $P$. 
We initiate a systematic study of nonnegative polynomials $P$ such that $P^k$ is not a sum of squares for any odd positive integer $k$. We call such polynomials \emph{stubborn}. Currently not much is known about stubborn polynomials, except for some isolated examples. One of our main results is that all extreme rays of the convex cone of nonnegative ternary sextics (homogeneous polynomials in $3$ variables of degree $6$) are stubborn. More generally, for a nonnegative ternary form $P$ we develop a new invariant $\delta^{\,\textrm{sos}}$ of real singularities of $P$, which we call \emph{a sum of squares invariant} or \emph{SOS-invariant}, such that if the sum of $\delta^{\,\textrm{sos}}$ over all real singularities of $P$ is sufficiently large, then $P$ is stubborn. This implies the result for ternary sextics, and allows us to construct stubborn forms in higher degrees. We compare the SOS-invariant to the classical delta invariant of a plane curve  singularity, and show that they agree for singularities of multiplicity $2$, but are not equal in general. We show that stubborn forms exist, for a fixed degree and number of variables, whenever nonnegative forms are not equal to sums of squares. We also prove that the set of nonnegative forms that are not stubborn is a convex cone, which includes the interior of the cone of nonnegative forms. We now go into more details and review the history of the problem.

For positive integers $n, d$, let $F_{n,d}$ denote the space of real forms (homogeneous polynomials) of degree $d$ in $n$ variables. We note that for analyzing questions about nonnegativity and sums of squares it suffices to consider homogeneous polynomials, as homogenization and dehomogenization preserve the properties of being nonnegative and being a sum of squares. From now on we will work with forms.
A form $P\in F_{n,d}$ is said to be \emph{nonnegative} if $P(\X)\geq 0$ for any $\X\in \R^n$.
If $P(\X)>0$ for all $\X\neq \mathbf{0}$, then $P$ is called \emph{strictly positive}.
If $P=H_1^2+\dots+H_r^2$ for some $H_1,\dots, H_r\in F_{n,d/2}$ of degree $d/2$, then $F$ is said to be a \emph{sum of squares}.
Trivially, every sum of squares is nonnegative and the degree $d$ of a nonnegative form must be even.
Following Choi and Lam \cite{CL}, let $P_{n,d}$ and $\Sigma_{n,d}$ denote the closed convex
cones of nonnegative forms and, respectively, sum of squares forms in $F_{n,d}$.
The interior $\textrm{int}(P_{n,d})$ of $P_{n,d}$ consists exactly of strictly positive forms of degree $d$. 
Let $\Delta_{n,d}:=P_{n,d}\setminus \Sigma_{n,d}$ be the difference of the two cones.
Hilbert \cite{Hilbert1888berDD} proved that $\Delta_{n,m} \neq \emptyset$ if and only if  
$n \ge 3$ and $d \ge 6$ or  $n \ge 4$ and $d \ge 4$, see Figure \ref{fig:table}.
\emph{The Motzkin form} $M\in F_{3,6}$ is a ternary sextic
\begin{align}\label{eq:Motzkin}
M\ =\ X_1^4X_2^2 +X_1^2X_2^4 + X_3^6 - 3X_1^2X_2^2X_3^2  
\end{align}
that was the first explicit example of a nonnegative form that is not a sum of squares \cite{Motzkin}.
Another early example of a form in $\Delta_{3,6}$ was
\begin{equation}\label{eq:Robinson}\begin{aligned}
R\ =&\ X_1^6 +X_2^6 + X_3^6 + 3X_1^2X_2^2X_3^2\\ &- (X_1^4X_2^2 + X_1^2X_2^4+X_1^4X_3^2+X_1^2X_3^4+X_2^4X_3^2+ X_2^2X_3^4), 
\end{aligned}\end{equation}
found by Robinson in \cite{Robinson1973}.

Hilbert's 17th Problem asks whether, for $P \in P_{n,d}$, there exists 
 $Q$ in some $F_{n,d'}$ so that $Q^2P \in \Sigma_{n,d+2d'}$. In 1927 Artin  \cite{Artin1927berDZ} solved this problem in the affirmative, even in a more general setting of \emph{real closed fields}. In equivalent terms, any nonnegative form can be written as a sum of squares of rational functions. 
 Later, multiple authors studied variations of Hilbert's 17th problem \cite{polya1928positive, Habicht, Delzell, Aspects}.
 In particular, for a given strictly positive form $Q$ it was desirable to know  whether for all $P\in P_{n,d}$ the form $PQ^k$ is a sum of squares for some large $k\geq 1$.
 Reznick showed \cite{Uniform} that any strictly positive form $P$ multiplied by a large enough power $Q^k$ for $Q=X_1^2+\dots+X_n^2$ is a sum of squares (this does not hold for all nonnegative forms $P$ \cite{Absence}). More generally, Scheiderer proved \cite{Sch} that if $Q \in \textrm{int}(P_{n,d'})$ and $P\in\textrm{int}(P_{n,d})$ are two strictly positive  forms, then $PQ^k\in \Sigma_{n,d+d'k}$ is a sum of squares for all sufficiently large $k$.
Thus, if $P\in \textrm{int}(P_{n,d})$ is strictly positive, then  $P^k\in \Sigma_{n,kd}$ is a sum of squares for all sufficiently large $k$.
In the present work we study this property for non-sum of squares forms $P\in \partial P_{n,d}$ in the boundary of the cone of nonnegative forms. Being a square, an even power $P^{2k}=(P^k)^2\in \Sigma_{n,2kd}$ of $P$ is a sum of squares. We say that $P$ admits an odd sum of squares power, if $P^{2k+1}\in \Sigma_{n,(2k+1)d}$ is a sum of squares for some $k\geq 0$.
Otherwise, the form $P\in \partial P_{n,d}$ will be called \emph{stubborn}.


A nonnegative form $P\in P_{n,d}$ is said to be \emph{extremal}, if it spans an extreme ray of the cone $P_{n,d}$. The set of extremal forms in $P_{n,d}$ is denoted by $\cE(P_{n,d})$. When $P\in \cE(P_{n,d})$ spans an exposed extreme ray, we say that $P$ is \emph{exposed extremal form}.
The Motzkin form \eqref{eq:Motzkin} is an example of an non-exposed extremal form in $P_{3,6}$ (see \cite[p. $8$]{CL2}, \cite[Thm. $5$]{R} and the proof of \cite[Thm. $2$]{blekherman_hauenstein_ottem_ranestad_sturmfels_2012}), while the Robinson form \eqref{eq:Robinson} is an exposed extremal ternary sextic (see \cite[Thm. $3.8$]{CL}).
 In \cite{CL, CL2} Choi and Lam also studied the following ternary sextic and quaternary quartic:  
 \begin{equation}
   \label{eq:S,Q}
   \begin{aligned}
     S\ &=\ X_1^4X_2^2+X_2^4X_3^2+X_3^4X_1^2-3X_1^2X_2^2X_3^2,\\
     Q\ &=\ X_4^4+X_1^2X_2^2+X_1^2X_3^2+X_2^2X_3^2-4X_1X_2X_3X_4.
   \end{aligned}
 \end{equation}
 They showed that $S\in \Delta_{3,6}$, $Q\in \Delta_{4,4}$ are non-sum of squares extremal nonnegative forms.
 In 1979 Stengle \cite{Sten} proved that the ternary sextic 
 \begin{equation}
   \label{eq:Stengle}
   \begin{aligned}
 T\ =\ X_1^3X_3^3 + (X_2^2X_3 - X_1^3 - X_1X_3^2)^2     
   \end{aligned}
 \end{equation}
 is stubborn.
 The paper \cite{Sten} reports that Reznick had proved that $S$ is stubborn by a different argument.
 In 1982, Choi, Dai, Lam and Reznick \cite{CDLR} cited Stengle's example \eqref{eq:Stengle} and claimed the same property for $M$ instead of $S$.
 No proofs for $M$ or $S$ were given at the time.
 In Subsection \ref{sec:Motzkin} we include this earlier proof of the fact that $M$ is stubborn.

 This work was in particular motivated by a query from Jim McEnerney about references for these claimed results.
 In his talk at the Conference on Applied Algebraic Geometry (AG23) held in Eindhoven in July 2023 Reznick posed a conjecture that all extremal forms in $\Delta_{3,6}$ are stubborn.
 In the present work we settle this conjecture.

\begin{theorem}\label{thm:main}
    Let $P\in \cE(P_{3,6})\cap \Delta_{3,6}$ be an extremal nonnegative ternary sextic which is not a sum of squares. Then $P$ is stubborn, i.e. $P^{2k+1}\in \Delta_{3,6(2k+1)}$ is not a sum of squares for $k\geq 0$.
\end{theorem}

As a direct consequence of this result we have the following.

\begin{corollary}\label{cor:main}
The forms $M$, $R$ and $S$ are all stubborn, i.e, $M^{2k+1}, R^{2k+1}, S^{2k+1} \in \Delta_{3,6(2k+1)}$ are not sums of squares for all $k\geq 0$.
\end{corollary}

\begin{remark}
  Theorem \ref{thm:main} does not imply Stengle's result above, as the ternary sextic $T\in \Delta_{3,6}$ is not extremal (see Subsection \ref{sub:Stengle}).
  Furthermore, by a result of Scheiderer \cite{Sch} that we stated above, all sufficiently large powers $P^{2k+1}$ 
  of a strictly positive form $P\in \textrm{\normalfont{int}}(P_{n,d})$ are sums of squares.
\end{remark}

\smallskip

More generally, we develop a new invariant of a real zero of a nonnegative ternary form, that we call  \emph{the SOS-invariant}.
It can be compared to the classical \emph{delta invariant} of a plane curve singularity (see Subsection \ref{sub:delta}). The main idea is that if the sum of SOS-invariants of $P\in P_{3,d}$ over all its real zeros is too large (specifically, greater than $d^2/4$), then $P$ must be stubborn. This happens, for example, for extremal ternary sextics in $\Delta_{3,6}$.
For higher degrees such forms exist due to results of Brugallé et al. from \cite{Brugall2018RealAC}. However, Theorem \ref{thm:main} does not admit a direct generalization, as no characterization of extremal forms in $P_{n,d}$ in terms of the number of real zeros is known except for $(n,d) = (3,6)$. 
\smallskip

By regarding a form $P\in P_{3,d}\subset P_{n,d}$ with more than $d^2/4$ real zeros as a form in $n\geq 4$ variables we show that stubborn forms exist in arbitrary number of variables, see Theorem \ref{thm:general}.
In Section \ref{sec:higher_n} we also show that the quaternary quartic $Q\in \Delta_{4,4}$ defined in \eqref{eq:S,Q} and \emph{the Horn form} $F\in \Delta_{5,4}$:
\begin{align}\label{eq:Horn}
F\ =\ \left(\sum_{j=1}^5 X_j^2\right)^2 - 4\ \sum_{j=1}^5 X_j^2X_{j+1}^2,
\end{align}
are both stubborn.
\begin{remark}
  The Horn form was originally defined as a quadric in $X_1^2,\dots, X_5^2$.
 It was communicated to Hall by Horn in the early 1960s, as a counterexample to a conjecture of Diananda asserting that a quadratic form that is nonnegative on the nonnegative orthant is a sum of a nonnegative form and a quadratic form with nonnegative coefficients only (see \cite[p.25]{Diananda1962OnNF} and \cite[p.334-5]{Hall1963CopositiveAC}).
\end{remark}

\smallskip

In Section \ref{sec:convexity} we initiate a systematic study of the set of non-stubborn forms that admit odd sums of squares powers.
For $k\geq 0$ let us define
\begin{align}\label{eq:Sigma_k}
\Sigma_{n,d}(2k+1)\ =\ \left\{P \in P_{n,d} \, :\, P^{2k+1} \in \Sigma_{n,d(2k+1)} \right\}.
\end{align}
Note that $\Sigma_{n,d}(1) = \Sigma_{n,d}$ and,
since $P^{2k+3} = P^{2k+1}\cdot P^2$, we have the inclusion $\Sigma_{n,d}(2k+1) 
\subseteq \Sigma_{n,d}(2k+3)$ and so it makes sense to define 
\begin{align}\label{eq:Sigma_inf}
\Sigma_{n,d}(\infty)\ =\ \bigcup_{k\,\geq\, 0} \Sigma_{n,d}(2k+1).
\end{align}
Let $\Delta_{n,d}(\infty)=P_{n,d}\setminus \Sigma_{n,d}(\infty)$ denote the set of stubborn forms in $P_{n,d}$.
With these notations, we have that $M, R, S, T \in \Delta_{3,6}(\infty)$, $Q\in \Delta_{4,4}(\infty)$ and $F\in \Delta_{5,4}(\infty)$. 
Since both $\Sigma_{n,d}$ and $P_{n,d}$ are closed convex cones, it is natural to ask this question for $\Sigma_{n,d}(2k+1)$ and for $\Sigma_{n,d}(\infty)$.

First observe that if $(P_i) \subset \Sigma_{n,d}(2k+1)$ is a sequence of forms converging to $P=\lim_{i\rightarrow \infty} P_i$,
then the sequence of powers $(P_i^{2k+1})\subset \Sigma_{n,d(2k+1)}$ converges to $P^{2k+1}=\lim_{i\rightarrow \infty} P^{2k+1}_i$, which by closedness of the sums of squares cone means that $P\in \Sigma_{n,d}(2k+1)$ and so $\Sigma_{n,d}(2k+1)$ is closed. It is unclear whether $\Sigma_{n,d}(2k+1)$ is convex when $k >0$. That is, if 
$P_1^{2k+1}$ and $P_2^{2k+1}$ are sums of squares, must  $(P_1+P_2)^{2k+1}$ be a sum of squares as well?
We prove in Theorem \ref{3.1} that if $P_1^{2k+1} $ is a sum of squares and $P_2$ is a sum of squares, then
$(P_1+P_2)^{2k+1}$ is a sum of squares. This is a special case of a more general Theorem \ref{3.3} that in particular yields convexity of \eqref{eq:Sigma_inf}.
Note however that $\Sigma_{n,d}(\infty)$ is not closed, when $\Delta_{n,d}(\infty) \neq \emptyset$, that is, if $\Delta_{n,d}\neq \emptyset$ (cf. Theorem \ref{thm:general}).
Indeed, a form $P\in\Delta_{n,d}(\infty)=P_{n,d}\setminus  \Sigma_{n,d}(\infty)$ lies in the closure of the open cone $\textrm{int}(P_{n,d}) \subset \Sigma_{n,d}(\infty)$ of strictly positive forms, each of which admits an odd power which is a sum of squares by \cite{Sch}.
Thus, the Motzkin form \eqref{eq:Motzkin} can be obtained as the limit $M=\lim_{\ep\rightarrow 0+} M_\ep$, where $ M_\ep = M + \ep (X_1^2 + X_2^2 + X_3^2)^3$ is strictly positive for $\ep>0$.
By Theorem \ref{3.1}, we see that  $\{ \ep\geq 0: M_\ep \in \Sigma_{3,6}(2k+1)\}$ is an interval of the form $[\be_{2k+1},\infty)$ for some $\be_{2k+1}>0$.
The coefficient of $X_1^2X_2^2X_3^2$ in $M_\ep$ is $-3+6\ep$, so $\be_1 \le \frac 12$.
Furthermore, one has $\be_1\geq \be_3\geq \be_5\geq \dots$ and $\lim_{k\rightarrow \infty}\be_{2k+1} = 0$.
 Thus, there are infinitely many $k$ so that $\Sigma_{3,6}(2k+1) \subsetneq \Sigma_{3,6}(2k+3) $. We strongly believe that this is true for all $k\geq 0$.


\begin{remark}\label{rem:hom}
The property of being nonnegative or a sum of squares is invariant under the (de)homogenization of a polynomial.
Thus, $P^{2k+1}$ is a sum of squares for  $P\in P_{n,d}$ if and only if so is $p^{2k+1}$ for the dehomogenized polynomial $p(x_1,\dots, x_{n-1}):=P(x_1,\dots, x_{n-1},1)$.
In particular, stubborn forms exist in $P_{n,d}$ if and only if there are \emph{stubborn} nonnegative $(n-1)$-variate polynomials of degree $d$.
\end{remark}

Based on Theorem \ref{thm:main} and Theorem \ref{thm:Q} from Section \ref{sec:higher_n} we make the following conjecture.

\begin{conjecture}\label{conj:general}
 Let $P \in \cE(P_{n,d})\cap \Delta_{n,d}$ be an extremal nonnegative form which is not a sum of squares. Then $P\in\Delta_{n,d}(\infty)$ is stubborn.
\end{conjecture}

\medskip

\noindent{\bf Acknowledgements:} We thank Pablo Parrilo, Adam Parusi\'nski and Isabelle Shankar for useful discussions. Moreover, we thank Iosif Pinelis for his proof of Theorem \ref{MO} and Jim McEnerney, whose query to the third author partially motivated this work.
 
\section{Preliminaries}

Here we collect definitions and prove auxiliary results that are used throughout the text.

\subsection{Order of vanishing}

A form $F\in F_{n,d}$ has \emph{order of vanishing} or, simply, \emph{multiplicity} at least $m\in \mathbb{N}$ at $\X^*\in \mathbb{P}_{\mathbb{C}}^{n-1}$, if $\partial_{\X}^{\,\alpha} F(\X^*)=0$ for all $\alpha\in \N^n$ with $\vert\alpha\vert=\alpha_1+\dots+\alpha_n\leq m-1$. This is equivalent to vanishing of directional derivatives of order up to $m-1$,
\begin{align}\label{eq:mult}
    \frac{\dd^i}{\dd\varepsilon^i}\bigg|_{\varepsilon=0} F(\X^*+\varepsilon \bV)\ =\ 0\quad\textrm{for all} \quad \bV\in \mathbb{C}^n\quad \textrm{and}\quad i=0,1,\dots, m-1.
\end{align}
In particular, $F$ has multiplicity at least $1$ at its zeros $\X^*\in \mathcal{V}(F)\subset \mathbb{P}_{\mathbb{C}}^{n-1}$ and multiplicity at least $2$ at singular points of the hypersurface $\mathcal{V}(F)\subset \mathbb{P}_{\mathbb{C}}^{n-1}$.
If $m$ is the largest integer satisfying \eqref{eq:mult}, we say that the multiplicity of $F$ at $\X^*$ is \emph{(exactly) $m$} and write $m_{\X^*}(F)=m$.

If $F\in F_{n,d}$ has multiplicity $2$ at $\X^*\in \mathbb{P}_{\mathbb{C}}^{n-1}$ and 
\emph{the Hessian matrix}
\begin{align*}\textrm{Hess}_{\X^*} F\ :=\ \left(\partial_{x_i}\partial_{x_j} F(\X^*)\right)_{i,j=1}^n
\end{align*}of $F$ at $\X^*$ is of maximal rank $\textrm{rk}\,( \textrm{Hess}_{\X^*} F) = n-1$,
\footnote{Due to the Euler's identity for homogeneous functions, any vector proportional to $\X^*$ lies in the kernel of $\textrm{\normalfont{Hess}}_{\X^*}(F)$ and hence the Hessian of $F$ at a singular point is always rank-deficient.}
one says that $\X^*$ is \emph{an ordinary singularity} of $\mathcal{V}(F)$. 
A real zero $\X^*\in \RP^{n-1}$ of a nonnegative form $P\in P_{n,d}$ is a singular point of $\mathcal{V}(P)\subset \mathbb{P}_{\mathbb{C}}^{n-1}$. If $\X^*$ is an ordinary singularity, it is sometimes called \emph{a round zero} of $P$ \cite{blekherman_hauenstein_ottem_ranestad_sturmfels_2012, Iliman2014}. Then, the Hessian matrix $\textrm{Hess}_{\X^*} P$ must be positive semidefinite of corank one. 
\begin{remark}\label{rem:m-even}
The multiplicity $m=m_{\X^*}(P)$ of a nonnegative $P\in P_{n,d}$ at $\X^*\in \RP^{n-1}$ is even.
Indeed, otherwise the restriction
\[P(\X^*+t \bV)\ =\ \frac{1}{m!} \frac{\dd^m}{\dd\varepsilon^m}\bigg|_{\varepsilon =0} P(\X^*+\varepsilon \bV)\, t^m+O(t^{m+1}) \]
of $P$ to some line through $\X^*$ cannot be nonnegative. 
\end{remark}

Given a real zero $\X^*\in \RP^{n-1}$ of $P\in P_{n,d}$, Reznick \cite{Reznick2007OnHC} considers a subspace 
\begin{align*}
E(P,\X^*)\ =\ \left\{Q\in F_{n,d/2}\,:\, P-\varepsilon Q^2\geq 0\ \textrm{in some neighborhood of}\ \X^*\ \textrm{for some}\ \varepsilon>0\right\}
\end{align*}
of forms of half degree whose square (up to a constant) is upper-bounded by $P$ locally around $\X^*$, where the topology is (induced by) the Euclidean one.
In the following we refer to $E(P,\X^*)$ as \emph{the local SOS-support} of $P$ at $\X^*$ 
We call the codimension of this linear subspace of $F_{n,d/2}$ the \emph{half-degree invariant} of $P$ at $X^*$ and denote it by
\begin{align}\label{eq:hd}
  \delta^{\,\textrm{hd}}(P,\X^*)\ =\ {n-1+d/2 \choose d/2} - \dim E(P,\X^*).  
\end{align}
For a nonnegative form $P\in P_{n,d}$ the total sum of  $\delta^{\,\textrm{hd}}(P,\X^*)$ over all zeroes of $X^*$ of $P$ is called \emph{the half-degree invariant} of $P$ and denoted by $\delta^{\,\textrm{hd}}(P)$,

The following lemma is rather simple.

\begin{lemma}\label{lem:vanishing}
Let $\X^*\in \RP^{n-1}$ be a real zero of a nonnegative form $P\in P_{n,d}$, let $k\in \N$ and let $H\in E(P^k,\X^*)$.
Then $m_{\X^*}(P^k)\geq 2k$ and $m_{\X^*}(H)\geq k$. 

\end{lemma}

\begin{proof}
Since the multiplicity of $P\in F_{n,d}$ at its real zero $\X^*\in \RP^{n-1}$ is even (see Remark \ref{rem:m-even}), the univariate polynomial $t\mapsto P(\X^*+t\bV)$ is divisible by $t^2$ for any $\bV\in \R^n$.
    Then $P^k(\X^*+t \bV)$ is divisible by $t^{2k}$, which means that all directional derivatives of $P^k$ at $\X^*$ of order less than $2k$ are equal to zero.

For $H\in E(P,\X^*)$, $\bV \in \R^n$ and a sufficiently small $t\in \R$, the  univariate polynomial $H^2(\X^*+t \bV)$ is bounded (up to a constant) by $P^k(\X^*+t\bV )= O(t^{2k})$. Then the multiplicity of $H$ at $\X^*$ is at least $k$. 
\end{proof}

\subsection{Intersection multiplicity}
For two bivariate polynomials $f, g\in \C[x_1,x_2]$ and a point $\x^*=(x^*_1,x^*_2)\in \A^2_{\C}$ \emph{the intersection multiplicity} of $f$ and $g$ at $\x^*$ is defined as the dimension of the quotient of \emph{the local ring} $\cO_{\x^*}=\left\{\frac{p}{q}\,:\, p,q\in \C[x_1,x_2],\, q(\x^*)\neq 0\right\}$ of $\A^2_{\C}$ at $\x^*$ by the ideal generated by $f$ and $g$,
\begin{align}\label{eq:local_int_mult}
    I_{\x^*}(f,g)\ =\ \dim_{\C}\left( \cO_{\x^*}/\langle f,g\rangle\right).
\end{align}
In particular,  we have that $I_{\x^*}(f,g)>0$ is positive if and only if $f(\x^*)=g(\x^*)=0$.
In this case, $I_{\x^*}(f,g)=1$ if and only the curves $f=0$ and $g=0$ intersect transversally at $\x^*$, that is, the gradient vectors $(\partial_{x_1} f(\x^*), \partial_{x_2} f(\x^*))$ and $(\partial_{x_1} g(\x^*), \partial_{x_2} g(\x^*))$ are linearly independent.
If $f$ and $g$ share a common factor in $\C[x_1,x_2]$ that vanishes at $\x^*$, the intersection multiplicity $I_{\x^*}(f,g)=\infty$ is infinite.
For ternary forms (homogenous polynomials in three variables) $F, G\in \C[X_1,X_2,X_3]$ and a point $\X^*\in \PP_{\C}^2$
one defines the intersection multiplicity of $F$ and $G$ at $\X^*$ as $I_{\X^*}(F,G):=I_{\x^*}(f,g)$, where $f$ and $g$ are dehomogenizations of $F$ and $G$, and $\x^*\in \A_{\C}^2$ is the representative of $\X^*\in \PP_{\C}^2$ in the corresponding affine chart.
The celebrated B\'ezout theorem asserts that the number of intersection points of two projective plane curves counted with multiplicities is equal to the product of their degrees.

\begin{theorem}[B\'ezout theorem]\label{thm:Bezout}
    Let $F, G\in \C[X_1,X_2,X_3]$ be ternary forms that have no common factors of positive degree. Then
    \begin{align}
       \sum_{\x^*\in \cV(F)\,\cap \,\cV(G)} I_{\x^*}(F,G)\ =\ \deg(F)\cdot \deg(G).
    \end{align}
    In particular, $\cV(F)\cap \cV(G)\subset\PP_{\C}^2$ consists of at most $\deg(F)\cdot \deg(G)$ points. 
\end{theorem}


A tangent to $\cV(F)\subset \mathbb{P}^2_{\mathbb{C}}$ at $\X^*\in \cV(F)$
is a projective zero $\mathbf{X}'\in \mathbb{P}\left(\X^* \right)^{\perp}$ of the homogeneous part of $\mathbf{X}'\in \left(\X^*\right)^\perp\mapsto F(\X^*+ \mathbf{X}')$ of lowest degree $m_{\X^*}(F)$.
In particular, tangents to $\cV(F)$ at $\X^*=[0:0:1]\in \cV(F)$ are projective zeros $\mathbf{X}'=[X'_1:X'_2]\in \mathbb{P}^1_{\mathbb{C}}$ of the degree $m_{\x^*}(f):=m_{\X^*}(F)$ part of the dehomogenized polynomial $f(X'_1,X'_2)=F(X'_1,X'_2,1)$. 
A curve $\cV(F)\subset\PP_{\C}^2$ can have at most $m_{\X^*}(F)$ linearly independent tangents. 
The following known result inspired our proof of Theorem \ref{thm:main}; this is essentially \cite[Thm. $3.4$]{Liang2019}.

\begin{lemma}\label{lem:multiplicities}
For ternary forms  $F, G\in \C[X_1,X_2,X_3]$ and $\X^*\in \PP_{\C}^2$,
\begin{align}
    I_{\X^*}(F,G)\ \geq\ m_{\X^*}(F)\cdot m_{\X^*}(G)
\end{align}
with equality if and only if $\cV(F)$ and $\cV(G)$ do not share a tangent at $\X^*$.
\end{lemma}

The blow-up of $\mathbb{A}^2_{\mathbb{C}}$ in a point $\x^*=(x_1^*,x_2^*)$ is a surface $S\subset  \mathbb{A}^2_{\mathbb{C}}\times \mathbb{P}^1_{\mathbb{C}}$ defined by the polynomial $(x_1-x_1^*)X'_2-(x_2-x_2^*)X'_1$ (where $[X_1':X_2']$ are the homogeneous coordinates on $\mathbb{P}^1_{\mathbb{C}}$) together with birational morphism
\begin{align}\label{eq:proj}
  \pi: S\ \rightarrow\ \mathbb{A}^2_{\mathbb{C}},\quad \left((x_1,x_2), [X_1':X_2']\right)\ \mapsto\ (x_1,x_2).
\end{align}
Points in $S$ with $[X'_1:X'_2]=[1:x'_2]$ (respectively, $[X'_1:X'_2]=[x'_1:1]$) form a Zariski open set $S_1$ (respectively, $S_2$) isomorphic to $\mathbb{A}^2_{\mathbb{C}}$ with coordinates $(x_1,x'_2)$ (respectively, $(x_2,x'_1)$). 
\emph{The blow-up} (or \emph{the total transform}) of an affine curve $f=0$ in a point $\x^*\in\mathbb{A}^2_{\mathbb{C}}$ is the inverse image of $f=0$ under \eqref{eq:proj}, that is, it is a curve on $S$ defined by $\pi^*f = f\circ \pi$.
\emph{The strict transform} of $f=0$ in $\x^*$ is the closure of the inverse image of $\{f=0\}\setminus \{\x^*\}$, it coincides with the total transform unless $f(\x^*)=0$.
If $f(\x^*)=0$, the blow-up of $f=0$ is the union of its strict transform and \emph{the exceptional line} $\pi^{-1}(\x^*)=\{\x^*\}\times \mathbb{P}^1_{\mathbb{C}}$, which is just a copy of the projective line.
In a local chart ($S_0$ or $S_1$) of $S$, the strict transform of $f=0$ is a curve in $\mathbb{A}^2_{\mathbb{C}}$ whose defining polynomial we denote by $f'$.   
Points $[X'_1:X'_2]$ on $\mathbb{P}^1_{\mathbb{C}}\simeq \pi^{-1}(\x^*)$ at which the strict transform intersects the exceptional line are called \emph{the first order infinitely near points} of $f=0$ at $\x^*$, they are identified with tangents of $\mathcal{V}(F)$ at $\X^*=[x_1^*:x_2^*:1]$ via $[X'_1:X'_2]\mapsto [X'_1:X'_2:X_3']$, where $X'_3=-x_1^*X'_1-x_2^*X'_2$.

Given a first order infinitely near point  $\x'\in \pi^{-1}(\x^*)$ of $f=0$ at $\x^*$,
we can blow-up the strict transform $f'=0$ in $\x'$ again.
After doing finitely many successive blow-ups at singular points of $f=0$ and of its higher order strict transforms, we end up with a smooth curve.
This resolution of singularities process is guaranteed to terminate by \cite[Thm. 1.43]{Kollar}.

\begin{example}\label{ex:Stengle0}
  The polynomial $f=x_1^2+x_2^4-2x_1x_2^2+x_1^3+2x_1^4-2x_1^3x_2^2+x_1^6$ is singular at $(0,0)$ with $m_{(0,0)}(f)=2$.
  Its strict transform (in the coordinates $(x_1', x_2)$, $x_1=x_1'x_2$) is given by \[f'\ =\ x_1'^2+x_2^2-2x_1'x_2+x_1'^3x_2+2x_1'^4x_2^2-2x_1'^3x_2^3+x_1'^6x_2^4\] and the unique first order infinitely near point $(x_1',x_2)=(0,0)$ of $f=0$ at $(0,0)$ has multiplicity $m_{(0,0)}(f')=2$.
  The strict transform of $f'=0$ in $(0,0)$ (in the coordinates $(x_1'',x_2)$ with $x_1'=x_1''x_2$) is given by
  \[
f''\ =\ x_1''^2+1-2x_1''+x_1''^3x_2^2+2x_1''^4x_2^4-2x_1''^3x_2^4+x_1''^6x_2^8
\]
and the unique first order infinitely near point $(x_1'',x_2)=(1,0)$ of $f'=0$ at $(0,0)$ has  $m_{(1,0)}(f'')=2$.
The third blow-up $x_1''=1+x_1'''x_2$ reveals
\begin{align*}
f'''\ =\ x_1'''^2+(1+x_1'''x_2)^3+2(1+x_1'''x_2)^4x_2^2-2(1+x_1'''x_2)^3x_2^2+(1+x_1'''x_2)^6x_2^6.
\end{align*}
Since $f'''=0$ does not have singular points on the exceptional fiber $x_2=0$, the resolution of singularities process terminates after the third blow-up.
\end{example}


The following Noether's formula gives us a way to compute the local intersection multiplicity \eqref{eq:local_int_mult} by doing successive blow-ups, it is a refinement of Lemma \ref{lem:multiplicities}.

\begin{theorem}{\cite[Lemma $3.3.4$]{casas2000singularities}, \cite[Theorem $3.10$]{ChalChal2022}}\label{thm:Noether}
  Let $f, g\in \mathbb{C}[x_1,x_2]$ be polynomials that have no common factors of positive degree and let $\x^*\in \mathbb{A}^2_{\mathbb{C}}$ be their common zero. Then
  \begin{align*}
    I_{\x^*}(f,g)\ =\ m_{\x^*}(f)\cdot m_{\x^*}(g) + \sum_{\x'} I_{\x'}(f',g').
  \end{align*}
  where the sum is over common first order infinitely near points $\x'$ of $f=0$ and $g=0$ at $\x^*$.
\end{theorem}

We end this subsection with proving that blow-ups preserve nonnegativity of polynomials.

\begin{lemma}\label{lem:nonnegative_blow}
  Let $f\in \R[x_1,x_2]$ be a polynomial that is nonnegative locally around $\x^*\in \mathbb{A}^2_{\,\R}$. If $f(\x^*)=0$ and $[X_1':X_2']\in \mathbb{P}^1_{\R}$ is a real first order infinitely near point of $f=0$ at $\x^*$, then  the strict transform of $f=0$ at $\x^*$  is given by a polynomial that is nonnegative locally around $(\x^*, [X_1':X_2'])\in S$. 
\end{lemma}

\begin{proof}
  Let us assume without loss of generality that $\x^*=(0,0)$ and $[X_1':X_2']=[a:1]$, $a\in \R$.
  Then by Remark \ref{rem:m-even} the multiplicity $m_{\x^*}(f)$ is even and in the local chart on $S$ given by the coordinates $(x_1',x_2)$, $x_1=x_1'x_2$, we have that $f(x_1,x_2)=x_2^{m_{\x^*}(f)}f'(x_1',x_2)$.
  It follows that $f'$ is nonnegative locally around $(a,0)\in \mathbb{A}^2_{\,\R}$.
\end{proof}

\subsection{The delta invariant and the SOS-invariant}\label{sub:delta}

\emph{The (local) delta invariant} $\delta_{\x^*}(f)$ is a classical invariant of a singular point $\x^*\in \mathbb{A}^2_{\mathbb{C}}$ of an algebraic curve $f=0$ that can be defined as the dimension
  \begin{align}\label{eq:delta_dim}
    \delta_{\x^*}(f)\ =\ \dim_{\mathbb{C}}\left(\, \overline{\mathcal{O}_{f,\x^*}} \,\big/\, \mathcal{O}_{f,\x^*}\right) 
  \end{align}
of \emph{the integral closure} $\overline{\mathcal{O}_{f,\x^*}}$ of the local ring $\mathcal{O}_{f,\x^*}:=\mathcal{O}_{\x^*}/(f)$ of the curve $f=0$ at $\x^*$. We set $\delta_{\X^*}(F):=\delta_{\x^*}(f)$ for a form $F\in \C[X_1,X_2,X_3]$ and $\X^*\in \mathbb{P}^2_{\mathbb{C}}$, where $f$ is a dehomogenization of $F$ and $\x^*$ is the affine representative of $\X^*$. 

\begin{remark}\label{rem:gd}
  The sum $\delta(F)$ of $\delta_{\X^*}(F)$ over all singular points $\X^*\in \mathbb{P}^2_{\mathbb{C}}$ of a plane curve $\cV(F)$ is known as \emph{the (total) delta invariant}.
By the genus-degree formula \cite[Section 3.11]{CasasAlverozbMATH01497487}, $\delta(F)$ is equal to the defect between the geometric genus of $\cV(F)$ and its arithmetic genus.
\end{remark}

Similarly to Noether's formula, one can define the local delta invariant as $\delta_{\x^*}(f)=1$ in case of an ordinary singularity and otherwise as
\begin{align}\label{eq:delta}
  \delta_{\x^*}(f)\ =\ \frac{m_{\x^*}(f)(m_{\x^*}(f)-1)}{2} + \sum_{\x'} \delta_{\x'}(f'),
\end{align}
where the sum is over all first order infinitely near points $\x'$ of $f=0$ at $\x^*$.

\begin{example}\label{ex:Stengle}
 For the polynomial $f=x_1^3+(x_2^2-x_1^3-x_1)^2=x_1^2+x_2^4-2x_1x_2^2+x_1^3+2x_1^4-2x_1^3x_2^2+x_1^6$ from  Example \ref{ex:Stengle0} the formula \eqref{eq:delta} gives us
\begin{equation}\label{eq:Stengle_delta}
  \begin{aligned}
  \delta_{(0,0)}(f)\ &=\ 1+\delta_{(0,0)}(f')\ =\ 1+1+\delta_{(1,0)}(f'')\ =\ 1+1+1\ =\ 3,
\end{aligned}
\end{equation}
where $f'$ and $f''$ define strict transforms of $f=0$ and $f'=0$ at $(0,0)$.
  \end{example}

For a real $f\in \R[x_1,x_2]$ and $\x^*\in \mathbb{A}^2_{\mathbb{R}}$ we also define \emph{the real delta invariant}
\begin{align}
  \delta^{\,\R}_{\x^*}(f)\ =\ \frac{m_{\x^*}(f)(m_{\x^*}(f)-1)}{2} + \sum_{\x'\,-\, \textrm{real} } \delta_{\x'}(f'),
\end{align}
where the sum is over only real first order infinitely near points of $f=0$ at $\x^*$.

\smallskip
Finally, the quantity that plays a prominent role in our work is \emph{the SOS-invariant}  $\delta^{\,\textrm{sos}}_{\x^*}(f)$ of a real polynomial $f$ at a real zero $\x^*$. 
\begin{definition}
The SOS-invariant $\delta^{\,\textrm{\normalfont{sos}}}_{\x^*}(f)$ of a real polynomial $f$ at a real zero $\x^*$ is defined by $\delta^{\,\textrm{\normalfont{sos}}}_{\x^*}(f)=1$ in case of an ordinary singularity and 
\begin{align}\label{eq:sos-delta}
  \delta^{\,\textrm{\normalfont{sos}}}_{\x^*}(f)\ =\ \frac{m_{\x^*}^2}{4} + \sum_{\x'\,-\,\textrm{\normalfont{real}} } \delta^{\,\textrm{\normalfont{sos}}}_{\x'}(f')
\end{align}
in general, where again the sum is over real first order infinitely near points of $f=0$ at $\x^*$.
\end{definition}
It is easy to see that, for real $f$ and $\x^*$ one has that
\begin{align}
  \delta^{\,\textrm{sos}}_{\x^*}(f)\ \leq\ \delta^{\,\R}_{\x^*}(f)\ \leq\ \delta_{\x^*}(f)
\end{align}
and we set $\delta^{\,\textrm{sos}}_{\X^*}(F)=\delta^{\,\textrm{sos}}_{\x^*}(f)$, $\delta^{\,\R}_{\X^*}(F):=\delta^{\,\R}_{\x^*}(f)$ if $f\in \R[x_1,x_2]$ is the dehomogenization of a real ternary form $F\in \R[X_1,X_2,X_3]$ and $\x^*\in \mathbb{A}^2_{\R}$ is the affine representative of $\X^*\in \mathbb{P}^2_{\R}$.





\begin{remark}
  The notion of the SOS-invariant $\delta^{\,\textrm{\normalfont{sos}}}_{(0,0)}(f)$ makes sense, if one, more generally, considers  a convergent power series $f\in \R\{x_1,x_2\}$ at $(0,0)$.
\end{remark}

In Section \ref{sec:ternary} we show that for any coprime $H_1,H_2\in E(P,\X^*)$ in the local SOS-support of $P\in P_{n,d}$ we have that $\delta^{\,\textrm{sos}}_{\X^*}(P)\geq I_{\X^*}(H_1,H_2)$. We conjecture that the equality holds if $H_1,H_2\in E(P,\X^*)$ are chosen generically.

\section{Ternary forms}\label{sec:ternary}

The goal of this section is to prove Theorem \ref{thm:main} and investigate the question of stubbornness of ternary forms in general.
We start with some properties of the SOS-invariant that we introduced in Subsection \ref{sub:delta}. 

\subsection{The SOS-invariant and applications}

We show that the SOS-invariant  gives a lower bound on the intersection multiplicity of any two elements of the local SOS-support.

\begin{proposition}\label{propo:mult-bound}
  Let $P\in P_{3,d}$ be a nonnegative ternary form with a real zero $\X^*\in \mathbb{P}_{\mathbb{R}}^2$.
  Then for any $H_1, H_2\in E(P,\X^*)$ we have 
$
    I_{\X^*}(H_1,H_2) \geq   \delta^{\,\textrm{\normalfont{sos}}}_{\X^*}(P)
 $.
\end{proposition}


\begin{proof}
  We perform an induction on the number of blow-ups needed to reach a round zero. 
  If $\X^*=[0:0:1]\in \RP^2$ is a round zero of $P$, then $p(x_1,x_2)=P(x_1,x_2,1)=a x_1^2+2b x_1x_2+c x_2^2+\dots$ does not have real tangents (as the quadratic form $a x_1^2+2b x_1 x_2+ c x_2^2$ must be positive definite) and we have $\delta_{\mathbf 0}^{\,\textrm{\normalfont{sos}}}(p)=1$.
  Any $H_1, H_2\in E(P,\X^*)$ must vanish at $\X^*$ (see Lemma \ref{lem:vanishing}) and therefore $I_{\X^*}(H_1,H_2)\geq 1=\delta_{\mathbf{0}}^{\,\textrm{\normalfont{sos}}}(p)$.
  In general, we have by Theorem \ref{thm:Noether} that
  \begin{align}\label{eq:local123}
I_{\X^*}(H_1,H_2)\ =\    I_{{\bf 0}}(h_1,h_2)\ =\ m_{\mathbf{0}}(h_1)m_{\mathbf{0}}(h_2)+\sum_{\x'} I_{\x'}(h_1',h_2'),
  \end{align}
  where the sum is over all common first order infinitely near points $\x'$ of $h_1=0$ and $h_2=0$ at $\mathbf{0}$.
  Because $H_1, H_2$ are in $E(P,\X^*)$, we have by Lemma \ref{lem:vanishing} that $m_{\mathbf{0}}(h_1), m_{\mathbf{0}}(h_2)\geq m_{\mathbf{0}}(p)/2$.
Denote by $T(p,\mathbf{0})$ the set of real infinitely near points $\x'$ of $p=0$ at $\mathbf{0}$. Thus, disregarding all $\x'$ in \eqref{eq:local123} that do not belong to $T(p,\mathbf{0})$, yields
  \begin{align*}
    I_{\mathbf{0}}(h_1,h_2)\ &\geq\ \left(\frac{m_{\mathbf{0}}(p)}{2}\right)^2+\sum_{\x'\,\in\, T(p,\, \mathbf{0})}  I_{\x'}(h_1',h_2')\ \geq\ \frac{1}{4}m_{\mathbf{0}}(p)^2+\sum_{\x'\,\in\, T(p,\, \mathbf{0})} \delta_{\x'}^{\,\textrm{sos}}(p')\ =\ \delta^{\,\textrm{\normalfont{sos}}}_{\mathbf{0}}(p),
  \end{align*}
  where the last bound follows by the induction step, see \cite[Thm. 1.43]{Kollar}.
\end{proof}

One of the key properties of the SOS-invariant is that it behaves well under taking powers.

\begin{proposition}\label{propo:mult-powers}
  Let $P\in P_{3,d}$ be a nonnegative ternary form that has a real zero at $\X^*\in \mathbb{P}_{\mathbb{R}}^2$.
  Then for any $k\geq 1$ we have $  \delta^{\,\textrm{\normalfont{sos}}}_{\X^*}(P^k)=k^2\delta^{\,\textrm{\normalfont{sos}}}_{\X^*} (P)$. 
\end{proposition}

\begin{proof}
  It suffices to prove the claim for $\X^*=[0:0:1]$.
  As usual, denote by $p(x_1,x_2)=P(x_1,x_2,1)$ the dehomogenization of $P$.
  The strict transform $p'$ of $p$ satisfies $(p')^k=(p^k)'$.
Performing induction on the number of blow-ups needed to reach a real zero without real tangents, it is enough to prove the claim for the latter case.
But in that case $\delta^{\,\textrm{\normalfont{sos}}}(p^k)=\frac{1}{4}m_{\mathbf{0}}(p^k)^2=\frac{k^2}{4} m_{\mathbf{0}}(p)^2$, since the multiplicity is multiplied by $k$ under taking $k$-th powers.
\end{proof}

\subsection{Ternary Sextics}\label{sub:main}


Before passing to the proof of Theorem \ref{thm:main} on ternary sextics, we discuss some of their properties.
We start with a folklore result (see, e.g., \cite[Thm. 7.1]{Reznick2007OnHC}).

\begin{lemma}\label{lem:folklore}
  If $P\in P_{3,6}$ is reducible over $\C$, then it is a sum of squares.
\end{lemma}

\begin{proof}
If $P\in P_{3,6}$ is reducible over $\C$ but irreducible over $\R$, it is a sum of two squares by \cite[Lemma 3.1]{CLR1}.
Otherwise, let $Q\in \R[X_1,X_2,X_3]$ be a real factor of $P$.
If either $Q$ or $-Q$ (say, $Q$) is a nonnegative form of degree $d\in \{2,4\}$, then both $Q\in P_{3,d}=\Sigma_{3,d}$ and $P/Q\in P_{3,6-d}=\Sigma_{3,6-d}$ are sums of squares by a result of Hilbert \cite{Hilbert1888berDD}.
If the real $Q$ is sign indefinite, we obtain that $Q$ and hence also $P$ has infinitely many real zeros by \cite[Prop. 2.5]{CLR1}.
But then \cite[Thm. 3.7]{CLR1} implies that $P\in \Sigma_{3,6}$ is a sum of squares.
\end{proof}

By considering $X_1^{2k}M(X_1,X_2,X_3)$, we see that this result is false for degrees greater than six, see \cite{CL2}. The following result can also be derived from \cite[Thm. 7.9]{Choi1980RealZO}, which appeared in earlier works of Djoković \cite{DJOKOVIC1976359}, Yakubovich \cite{Yakubovich}, Popov  \cite{popov1973hyperstability}, Rosenblum and Rovnyak \cite{Rosenblum1971TheFP}.
However, we give an alternative proof.

\begin{lemma}\label{lem:scroll}
  Let $P\in P_{3,6}$ be a nonnegative ternary sextic with a real zero $\X^*\in \PP_{\R}^2$ of multiplicity $m_{\X^*}(P)$ at least $4$. Then $P$ is a sum of squares.     
\end{lemma}

\begin{proof}
  Without loss of generality we can assume that $P$ has a zero at $\X^*=[0:0:1]$.
  Since we assume that its multiplicity at $\X^*$ is at least $4$, the Newton polytope of $P(x_1,x_2)=p(x_1,x_2,1)$ is contained in a trapezoid $\Delta:=\textrm{conv}((4,0),(0,4),(6,0),(0,6))$  on Figure \ref{fig:New1}.
  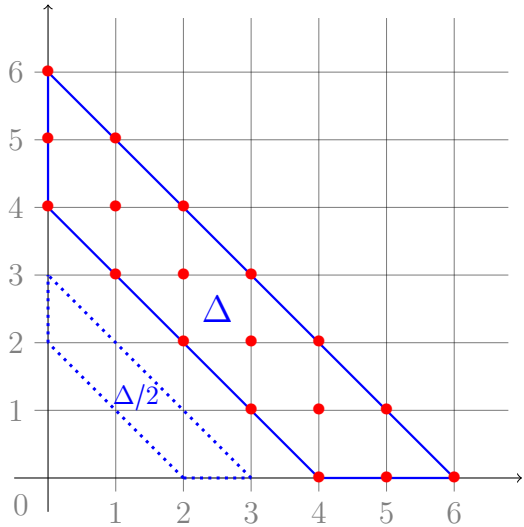
\begin{figure}[h!]
    \centering
  \begin{tikzpicture}[xscale=0.09,yscale=0.09,domain=0.125:220,samples=400]
    \node[opacity=0.5] at (-4,-4) {$0$};

    \draw[->] (-5,0) -- (70,0) node[below] {};
    \draw[->] (0,-5) -- (0,70) node[left] {};
    \foreach \i in {1,2,...,6} {
        \draw[opacity=0.5] (\i*10,68) -- (\i*10,-2) node[below] {$\i$};
    }
    \foreach \i in {1,2,...,6} {
        \draw[opacity=0.5] (68,\i*10) -- (-2,\i*10) node[left] {$\i$};
    }
    \draw[line width = 0.3mm, blue] (60,0) -- (0,60);
    \draw[line width = 0.3mm, blue] (40,0) -- (0,40);
    \draw[line width = 0.3mm, blue] (40,0) -- (60,0);
    \draw[line width = 0.3mm, blue] (0,40) -- (0,60);

    \node[blue] at (25,25) {{\large$\Delta$}};

    \draw[line width = 0.4mm, dotted, blue] (30,0) -- (0,30);
    \draw[line width = 0.4mm, dotted, blue] (20,0) -- (0,20);
    \draw[line width = 0.4mm, dotted, blue] (20,0) -- (30,0);
    \draw[line width = 0.4mm, dotted, blue] (0,20) -- (0,30);

    \node[blue] at (13,12) {{\footnotesize $\Delta/2$}};

    \draw[red] (40,0) node {$\bullet$};
    \draw[red] (50,0) node {$\bullet$};
    \draw[red] (60,0) node {$\bullet$};
    \draw[red] (30,10) node {$\bullet$};
    \draw[red] (40,10) node {$\bullet$};
    \draw[red] (50,10) node {$\bullet$};
    \draw[red] (20,20) node {$\bullet$};
    \draw[red] (30,20) node {$\bullet$};
    \draw[red] (40,20) node {$\bullet$};
    \draw[red] (10,30) node {$\bullet$};
    \draw[red] (20,30) node {$\bullet$};
    \draw[red] (30,30) node {$\bullet$};
    \draw[red] (0,40) node {$\bullet$};
    \draw[red] (10,40) node {$\bullet$};
    \draw[red] (20,40) node {$\bullet$};
    \draw[red] (0,50) node {$\bullet$};
    \draw[red] (10,50) node {$\bullet$};
    \draw[red] (0,60) node {$\bullet$};
  \end{tikzpicture}    
    \caption{The (maximal) support and the Newton polytope of a degree $6$ \\ polynomial having a zero of multiplicity $4$ at $(0,0)$.}
    \label{fig:New1}
  \end{figure}
  Let us consider vectors $\z=(z_{\Vector{\alpha}})_{\Vector{\alpha}\in \Delta/2}$ and $m_{\Delta/2}(\x)=\left(\x^{\Vector{\alpha}}\right)_{\Vector{\alpha}\in \Delta/2}$ of variables, respectively, monomials in $\x=(x_1,x_2)$, indexed by lattice points in $\Delta/2$. We also regard $m_{\Delta/2}$ as a monomial map $\x\mapsto \left(\x^{\Vector{\alpha}}\right)_{\Vector{\alpha}\in \Delta/2}$, the projective closure of whose image is a toric variety $X:=\overline{\{ m_{\Delta/2}(\x)\,:\, \x\in \mathbb{A}^2_{\,\mathbb{C}}\}}\subset \mathbb{P}_{\mathbb{C}}^{\vert\Delta/2\vert-1}$.
  The polynomial $p\in \R[x_1,x_2]$ (having support  contained in $\Delta$) can be regarded as a quadratic form $\z^{\mathsf T} Q\, \z$ restricted to (an open dense subset of) $X$ (cf. $p(\x)=m_{\Delta/2}(\x)^{\mathsf T} Q \,m_{\Delta/2}(\x)$). In particular, the quadratic form $\z^{\mathsf T} Q\,\z$ is nonnegative in the homogeneous coordinate ring $\R[X]:=\R[\z]\,\big/\,\mathcal{I}(X)$ of $X$.
By \cite[Example 6.2]{BSV}, $X$ is a \emph{variety of minimal degree}. Then \cite[Thm. 1.1]{BSV} implies that $\z^{\mathsf T} Q\, \z$ is a sum of squares in $\R[X]$, that is, $\z^{\mathsf T}Q \z=\sum_{i=1}^rH_i(\z)^2 \mod \mathcal{I}(X)$ for some linear forms $H_i\in \R[\z]$. Thus, $p(\x)=m_{\Delta/2}(\x)^{\mathsf T} Q\, m_{\Delta/2}(\x)=\sum_{i=1}^rH_i(m_{\Delta/2}(\x))^2$ is a sum of squares in $\R[\x]$.
\end{proof}

With the help of the above lemma we can demonstrate that for forms in $\Delta_{3,6}$, the delta invariant \eqref{eq:delta} agrees with the SOS-invariant defined in Subsection \ref{sub:delta}. This is a particular case of the following result.


\begin{lemma}\label{lem:delta=delta}
  Let $\X^*\in \mathbb{P}^2_{\mathbb{R}}$ be a real zero of $P\in P_{3,d}$ with $m_{\X^*}(P)=2$. Then $\delta_{\X^*}(P)=   \delta_{\X^*}^{\,\textrm{\normalfont{sos}}}(P)$.
\end{lemma}

\begin{proof}
  Assume without loss of generality that $\X^*=[0:0:1]$ and consider the dehomogenized polynomial $p(X_1,X_2)=P(x_1,x_2,1)$.

If $\x'$ is a real first order infinitely near point of $p=0$ at $\x^*=\bf 0$, formula \eqref{eq:delta} implies that $\delta_{\x^*}(p)=1+\delta_{\x'}(p')$ (recall that $m_{\X^*}(P)=2$).


Also, by its definition, the invariant $\delta^{\,\textrm{\normalfont{sos}}}_{\x^*}(p)$ decreases by $1$ after a blow-up, that is, $\delta^{\,\textrm{\normalfont{sos}}}_{\x^*}(p)=1+\delta^{\,\textrm{\normalfont{sos}}}_{\x'}(p')$.
Moreover, $p'$ is of multiplicity at most $2$ at $\x'$.
After a finitely many blow-ups we arrive at an ordinary singularity (round zero) for which we have $\delta^{\,\textrm{\normalfont{sos}}}_{\x'}(p)=1=\delta_{\x'}(p)$. The claim follows by induction on the number of blow-ups.
\end{proof}

We now show that for extreme sextics in $\cE(P_{3,6})\cap \Delta_{3,6}$ the half-degree invariant \eqref{eq:hd} agrees with the delta-invariant \eqref{eq:delta} and hence, by Lemma \ref{lem:delta=delta}, also with the SOS-invariant.

\begin{lemma}\label{lem:Bruce's_delta=delta}
  Let $P\in \cE(P_{3,6})\setminus \Sigma_{3,6}$.
  Then $\delta^{\,\textrm{\normalfont{hd}}}(P)=\delta(P)=\delta^{\,\textrm{\normalfont{sos}}}(P)=10$. 
\end{lemma}

\begin{proof}
  Let us again assume that $\X^*=[0:0:1]$ and $\x^*=(0,0)\in \mathbb{A}^2_{\mathbb{C}}$ is an isolated real zero of $p(x_1,x_2)=P(x_1,x_2,1)$.
  By Lemma \ref{lem:scroll}, $m_{\x^*}(p)=m_{\X^*}(P)=2$.

  Let us  observe that if $\x'$ is a real first order infinitely near point of $p=0$ at $\x^*$, then $\delta^{\,\textrm{hd}}(p,\x^*)\leq 1+\delta^{\textrm{hd}}(p',\x')$ decreases at most by $1$ with a blow-up. Here we denote $\delta^{\,\textrm{hd}}(p,\x^*):=\delta^{\,\textrm{hd}}(P,\X^*)$.
  For the delta invariant we have equality $\delta_{\x^*}(p)=1+\delta_{\x'}(p')$, see \eqref{eq:delta}.
  Since after finitely many blow-ups we reach a round zero for which it would hold $\delta_{\x^*}(p)=\delta^{\,\textrm{hd}}(p,\x^*)=1$ (see, for example, \cite[p. $24$]{Reznick2007OnHC}), we have that $\delta^{\,\textrm{hd}}(p,\x^*)\leq \delta_{\x^*}(p)$.

  By \cite[Thm. 2]{blekherman_hauenstein_ottem_ranestad_sturmfels_2012}, the curve $\cV(P)\subset\mathbb{P}_{\C}^2$ defined by  $P\in \cE(P_{3,6})\setminus \Sigma_{3,6}$ is rational and, by \cite[Rem. 8]{blekherman_hauenstein_ottem_ranestad_sturmfels_2012}, it is a limit of forms in $\cE(P_{3,6})\setminus \Sigma_{3,6}$ that span exposed rays of $P_{3,6}$.
  Therefore, by the genus-degree formula (see Remark \ref{rem:gd}) we have that $\delta^{\mathbb{R}}(P)=\sum_{\X\in \mathcal{V}_{\mathbb{R}}(P)} \delta_{\X}(P)=10$, and $\delta(P)=10$ as well.
  Also, $\delta^{\,\textrm{hd}}(P) \geq 10$ (cf. \cite[Thm. 7.8]{Reznick2007OnHC}) as we explain in the end of the proof.
Combining these inequalities together, we obtain
  \begin{align}
    10\ =\ \delta(P)\ =\ \sum_{\X\in \mathcal{V}_{\mathbb{R}}(P)} \delta_{\X}(P)\ \geq \sum_{\X\in \mathcal{V}_{\mathbb{R}}(p)} \delta^{\,\textrm{hd}}(P,\X)\ =\ \delta^{\,\textrm{hd}}(P)\ \geq\ 10
  \end{align}
and so $\delta^{\,\textrm{hd}}(P)=\delta(P)=\delta^{\,\textrm{sos}}(P)$, where the second equality follows from Lemmas \ref{lem:scroll} and \ref{lem:delta=delta}.

It remains to justify  $\delta^{\,\textrm{hd}}(P)=\sum_{\X\in \mathcal{V}_{\mathbb{R}}(P)} \delta(P,\X) \geq 10$.  If this is false, then
  \[\textrm{codim}\, \left(\bigcap_{\X\in \mathcal{V}_{\mathbb{R}}(P)} E(P,\X)\right)\ \leq\ \sum_{\X\in \mathcal{V}_{\mathbb{R}}(P)} \textrm{codim}\, E(P,\X)\ =\ \sum_{\X\in \mathcal{V}_{\mathbb{R}}(P)} \delta^{\,\textrm{hd}}(P,\X)\ <\ 10,\]
  and in the $10$-dimensional space of ternary cubic forms there must exist $H\in \bigcap_{\X\in \mathcal{V}_{\mathbb{R}}(P)} E(P,\X)$.
  This in particular means that for a small enough $\varepsilon>0$, a degree six form $P-\varepsilon H^2$ is globally nonnegative (that is, $P-\varepsilon H^2\in P_{3,6}$) and hence $P=(P-\varepsilon H^2)+\varepsilon H^2$ is not extremal.  
  \end{proof}

\begin{remark}\label{rem:count_proved}
  The proof of Lemma \ref{lem:Bruce's_delta=delta} yields the formula
  \begin{align*}
    \delta^{\,\textrm{\normalfont{hd}}}(P)\ =\ \sum_{\X\in \mathcal{V}_{\mathbb{R}}(P)} \delta^{\,\textrm{\normalfont{hd}}}(P,\X)\ =\ 10
    \end{align*}
for any $P\in \cE(P_{3,6})\setminus \Sigma_{3,6}$, which was conjectured  in \cite[Conj. 7.9]{Reznick2007OnHC}. 
\end{remark}

Now we are ready to prove our main Theorem \ref{thm:main}.\\
\begin{proof}[Proof of Theorem \ref{thm:main}]
  Let $P\in \cE(P_{3,6})\cap \Delta_{3,6}$. By Lemma \ref{lem:Bruce's_delta=delta} we know that $\delta^{\,\textrm{\normalfont{sos}}}(P)=10$.
  Let now $k\geq 1$ be odd. Then we have by Proposition \ref{propo:mult-powers} that
  \begin{align*}
    \delta^{\,\textrm{\normalfont{sos}}}(P^k)\ =\ \sum_{\X\in \cV_{\R}(P)} \delta^{\,\textrm{\normalfont{sos}}}_{\X}(P^k)\ =\ \sum_{\X\in \cV_{\R}(P)}k^2\delta^{\,\textrm{\normalfont{sos}}}_{\X}(P)\ =\ k^2 \,\delta^{\,\textrm{\normalfont{sos}}}(P)\ =\ 10\,k^2.
  \end{align*}
  If $P^k=\sum_{i=1}^r H_i^2$ was a sum of squares, then necessarily $r\geq 3$ (since otherwise $P^k=H_1^2+H_2^2=(H_1+i H_2)(H_1-i H_2)$ and hence $P\in \Delta_{3,6}$ is reducible, which is impossible by Lemma \ref{lem:folklore}).
Then again by Lemma \ref{lem:folklore} and \cite[Lemma $4.5$]{Choi1980RealZO} the degree $3k$ forms $H_1$ and $H:=a_2H_2+\dots+a_rH_r$ (for some $a_2,\dots, a_r\in \R$) are coprime and hence
$I(H_1,H)=(3k)\cdot(3k)=9k^2$ holds by Theorem \ref{thm:Bezout}. On the other hand, for any real zero $\X^*\in \cV_{\R}(P)$ of $P$, Proposition \ref{propo:mult-bound} gives $I_{\X^*}(H_1,H)\geq \delta^{\,\textrm{\normalfont{sos}}}_{\X^*}(P^k)$. Combining everything together, we obtain that
\begin{align*}
  9k^2\ =\ I(H_1,H)\ \geq\ \sum_{\X\in \cV_{\R}(P)} I_{\X}(H_1,H)\ \geq\ \sum_{\X\in \cV_{\R}(P)} \delta^{\,\textrm{\normalfont{sos}}}_{\X}(P^k)\ =\ \delta^{\,\textrm{\normalfont{sos}}}(P^k)\ =\ 10\,k^2,
\end{align*}
which is an obvious contradiction.

\end{proof}

\subsection{Ternary forms of higher degree}\label{sub:higher_degree}

We note that existence of stubborn ternary forms in degree $6$ implies existence of stubborn ternary forms of higher degree via the following well-known technique (see, for example, \cite[$(1.4)$]{CL2}).

\begin{proposition}\label{prop:ext}
Suppose that $F \in P_{n,d}$ is stubborn. Then $X_1^{2m}F\in P_{n,d+2m}$ is also stubborn for all nonnegative integers $m\geq 1$.
\end{proposition}
\begin{proof}
  It is clear that $X_1^{2m}F$ is nonnegative. Suppose that $(X_1^{2m}F)^k$ is a sum of squares for some odd $k$, that is,
    $X_1^{2km}F^{k} = \sum_{i=1}^r H_i^2$
  Then it follows that $X_1^{km}$ divides all $H_i$ , and therefore $F^k$ is a sum of squares. This contradicts the assumption that $F$ is stubborn.
\end{proof}

We can also use the same argument as in Theorem \ref{thm:main} to show that nonnegative ternary forms with sufficiently many zeros are stubborn, which leads to more interesting examples of stubborn ternary forms of higher degree.

\begin{theorem}\label{thm:higher_d}
  Let $P\in P_{3,d}$ be a nonnegative ternary form of degree $d$ with  $\delta^{\,\textrm{\normalfont{sos}}}(P)> d^2/4$. Then $P$ is stubborn.

\end{theorem}


\begin{proof}
  Assume $P$ is not stubborn, that is, $P^k=H_1^2+H_2^2+\dots$ for some odd $k\geq 1$.
  Then by Propositions \ref{propo:mult-powers} and \ref{propo:mult-bound} we have  $I_{\X^*}(H_1,H_2)\geq \delta^{\,\textrm{\normalfont{sos}}}_{\X^*}(P^k)=k^2\delta^{\,\textrm{\normalfont{sos}}}_{\X^*}(P)$ for every real zero $\X^*\in \mathbb{P}_{\mathbb{R}}^2$ of $P$. Summing up over all real zeros of $P$, we obtain  \[
   \left(\frac{dk}{2}\right)^2\ =\ I(H_1,H_2)\ \geq\ k^2\delta^{\,\textrm{\normalfont{sos}}}(P)\ >\ \frac{k^2d^2}{4},\]
  which is a contradiction.
\end{proof}

Since a real zero contributes to the SOS-invariant by at least one, we have a corollary.

\begin{corollary}\label{cor:higher_d}

  If $P\in P_{3,d}$ has more than $d^2/4$ isolated zeros in $\mathbb{P}_{\mathbb{R}}^2$, then $P$ is stubborn.
\end{corollary}


The next example shows that in general the delta invariant, the SOS-invariant and the half-degree invariant can be different.

\begin{example}
  The ternary octic
  \begin{align}\label{eq:Motzkin_change}
    P\ =\ X_1^4X_2^4X_3^6 M(1/X_1, 1/X_2,1/X_3)\ =\ X_1^4X_2^4+ X_1^2X_3^6+X_2^2X_3^6-3X_1^2X_2^2X_3^4
  \end{align} belongs to $\cE(P_{3,8})\cap \Delta_{3,8}$ \cite[p. 372]{R}.
  It has five round zeros at $[0:0:1]$, $[\pm 1,\pm 1,1]$ and two more degenerate zeros $\X^*$ at $[1:0:0]$ and $[0:1:0]$ (cf. \cite[p. 25]{Reznick2007OnHC}). Using \eqref{eq:hd}, \eqref{eq:sos-delta} and \eqref{eq:delta}, one computes $\delta^{\,\textrm{\normalfont{hd}}}(P,\X^*)=5$, $\delta^{\,\textrm{\normalfont{sos}}}_{\X^*}(P)=6$ and $\delta_{\X^*}(P)=8$ that gives $\delta^{\,\textrm{\normalfont{hd}}}(P)=15=5+2\cdot 5$, $\delta^{\,\textrm{\normalfont{sos}}}(P)=5+2\cdot 6=17$ and $\delta(P)=5+2\cdot 8=21$ for the total invariants. 
  Theorem \ref{thm:higher_d} implies that $P$ is stubborn (that is, $P\in \Delta_{3,8}(\infty)$).
  This can be also derived from the stubbornness of the Motzkin form and \eqref{eq:Motzkin_change}  which is valid when the roles of $M$ and $P$ are reversed.
  This gives further evidence to our Conjecture \ref{conj:general}.
  
\end{example}

\begin{remark}
  In \cite[Thm. $4.5$]{Brugall2018RealAC}, via \emph{combinatorial patchworking}, Brugall\'e et al. construct nonnegative forms $P\in P_{3,d}$ with more than $d^2/4$ isolated real zeros for any $d$.
 By Corollary \ref{cor:higher_d} all such forms are stubborn. 
\end{remark}


\subsection{The earlier proof for the Motzkin form}\label{sec:Motzkin}

Below we give the elementary proof (alluded to in \cite{Sten, CDLR}) that the Motzkin form is stubborn.
Recall that the Newton polytope of a polynomial $p=\sum_{\alpha} p_{\alpha} \x^\alpha$ written in the basis of monomials is a convex polytope $\New(p)=\textrm{conv}(\{\alpha\in \mathbb{Z}^n\,:\, p_\alpha\neq 0\})$.
The following property of the Newton polytope of a sum of squares form is well-known.
 
\begin{lemma}{\cite[Thm. 1]{R}}\label{R1}
If $p = \sum_{i=1}^r h_i^2$ is a sum of squares, then $2\cdot\New(h_i)\subseteq \New(p)$.
\end{lemma}

For the Motzkin form, this codifies the elimination of possible terms in a square. We have
\begin{align*}
  \New(M)\ =\ \textrm{conv}\left((4,2,0),(2,4,0),(0,0,6)\right)
\end{align*}
and, if $M = \sum_{i=1}^r H_i^2$ was a sum of squares, $\New(H_i)$, $i=1,\dots, r$, would be contained in a triangle
$\Delta:=\textrm{conv}\left((2,1,0),(1,2,0),(0,0,3)\right)$.
The only integer points in $\Delta$ are its vertices and the interior point $(1,1,1)$,
hence the only possible terms in $H_i$ are $X_1^2X_2, X_1X_2^2, X_3^3, X_1X_2X_3$.
Another useful fact relates to representations of an even form as a sum of squares, this result appeared in \cite{CLR1}.
\begin{proposition}{\cite[Thm.4.1]{CLR1}}\label{R2}
Suppose $P\in \Sigma_{n,d}$ is an even sum of squares form, then we may write $P = \sum_{i=1}^r H_i^2$, where $H_i = \sum_{\alpha} H_{i\alpha}\X^{\alpha}$ and all $\alpha$'s belong to
one congruence class modulo 2 componentwise.
\end{proposition}
The Motzkin form is even and no two of the monomials $X_1^2X_2, X_1X_2^2, X_3^3, X_1X_2X_3$ (corresponding to $(2,1,0)$,
 $(1,2,0)$, $(0,0,3)$ and $(1,1,1)$ respectively) belong to the same congruence class modulo 2. Thus, if $M$ was a sum of squares,
it would have to be of the form $M=c_1(X_1^2X_2)^2 + c_2(X_1X_2^2)^2 + c_3(X_3^3)^2 + c_4(X_1X_2X_3)^2$ with $c_1,c_2,c_3,c_4 \ge 0$, which is false.

The third fact is a special case of a result on the square of products of linear polynomials. 
\begin{lemma}\label{R3}
Suppose $h_1,\dots, h_r \in \mathbb{R}[t]$ and let $k \ge 1$.
If
\begin{align}\label{eq:identity}
  (t^2-1)^{2k}\ =\ \sum_{i=1}^r h_i(t)^2,
\end{align}
then each $h_i(t)$ is a multiple of $(t^2-1)^k$.
\end{lemma}

\begin{proof}
  The proof is by induction on $k$.
 If $k = 1$, then $0 = \sum_{i=1}^r h_i(\pm1)^2$, so $h_i(t) = \tilde h_i(t)\,(t^2-1)$ and after cancelling $(t^2-1)^2$, we have
 $1 = \sum_{i=1}^r \tilde h_i(t)^2$, so each $h_i$ must be a constant.
 In the inductive step, we just need to factor out $(t^2-1)^2$ from both sides of \eqref{eq:identity} and repeat the same argument.
\end{proof}

We now prove the stubbornness of $M$.
\begin{theorem}
The Motzkin form \eqref{eq:Motzkin} is stubborn, that is, $M^k$ is not a sum of squares for all odd $k\geq 1$.
\end{theorem}
\begin{proof}
Suppose that $M^k\in \Sigma_{3,6k}$ is a sum of squares and write
\begin{equation}\label{E:1}
M^{k}\ =\ \left(X_1^4X_2^2 + X_1^2X_2^4 + X_3^6 - 3X_1^2X_2^2X_3^2\right)^k\ =\ \sum_{i=1}^r H_i^2.
\end{equation}
We remark for later use that by taking $(X_1,X_2,X_3) = (0,0,1)$ above,
\begin{equation}\label{E:2}
1\ =\ M^{k}(0,0,1)\ =\ \sum_{i=1}^r H_i(0,0,1)^2\ =\ \sum_{i=1}^r H_{i,(0,0,3k)}^2,
\end{equation}
where $H_{i,(0,0,3k)}:=H_i(0,0,1)$ is the coefficient of $X_3^{3k}$ in $H_i$, $i=1,\dots, r$.
Evidently,
\begin{align*}
\New\left(M^{k}\right)\ =\ k\cdot\New(M)\ =\ \textrm{conv}\left((4k,2k,0),(2k,4k,0),(0,0,6k)\right)
\end{align*}
and so by Lemma \ref{R1} the monomials in each $H_i$ must be taken from the triangle
\[
\textrm{conv}\left((2k,k,0),(k,2k,0),(0,0,3k)\right).
\]
When is $(\alpha_1,\alpha_2,\alpha_3) \in \mathbb Z^3$ in this triangle? We have $\alpha_3 = 3k - \alpha_1 - \alpha_2 \ge 0$ and the other two edges give $2\alpha_1 \ge \alpha_2$ and $2\alpha_2 \ge \alpha_1$. Note that $\alpha_2 \le 2k$ and
if  $\alpha_2 = 2k$, then we must have $\alpha_1 = k$.
In particular, $\alpha_1$ must be odd. Further, if $\alpha_2 = 0$, then $\alpha_1=0$. If $\alpha_2= 2$, then $1 \le \alpha_1 \le 4$ and
if $\alpha_2 = 2k-2$, then $k-1 \le \alpha_1 \le k+2$.

By Proposition \ref{R2}, we may assume that each $H_i$ contains only terms $X_1^{\alpha_1}X_2^{\alpha_2}X_3^{\alpha_3}$ of a particular parity.
In view of the interest in $X_3^{3k}$ (cf. \eqref{E:2}),
consider those $H_i$'s in which $\alpha_1$
and $\alpha_2$ are even and $\alpha_3$ is odd. By the above remarks, we must have $\alpha_2\le 2k-2$, since $\alpha_1$ is even, and also the only term for which $\alpha_2= 0$
has $\alpha_1=0$ as well. So we can write such an $H_i$ in increasing powers of $X_2$ as 
\begin{equation}\label{E:4}
\begin{aligned}
  H_i\ =\ &H_{i,(0,0,3k)} X_3^{3k}
                    +  \left(H_{i,(2,2,3k-4)} X_1^2X_3^{3k-4} +  H_{i,(4,2,3k-6)} X_1^4X_3^{3k-6}\right) X_2^2\\ 
                   &+ \cdots 
                      + \left(H_{i,(k-1,2k-2,3)} X_1^{k-1}X_3^3 + H_{i,(k+1,2k-2,1)} X_1^{k+1}X_3\right) X_2^{2k-2}.
\end{aligned}
\end{equation}
Now observe that $M(1,t,1) = t^2 + t^4 + 1 - 3t^2 = (t^2-1)^2$. Therefore, \eqref{E:1} specializes to
\begin{align*}
  M^{k}(1,t,1)\ =\ (t^2-1)^{2k}\ =\ \sum_{i=1}^r H_i(1,t,1)^2.
\end{align*}
and by Lemma \ref{R3} we have $H_i(1,t,1)= c_i(t^2-1)^{k}$ for some $c_i\in \mathbb{R}$.
On the other hand, from \eqref{E:4} we obtain
\begin{align*}
H_i(1,t,1)\ =\ H_{i,(0,0,3k)} +(H_{i,(2,2,3k-4)}  +  H_{i,(4,2,3k-6)})t^2 + \cdots + (H_{i,(k-1,2k-2,3)} + H_{i,(k+1,2k-2,1)}) t^{2k-2}.
\end{align*}
and so $\deg (H_i(1,t,1)) =\deg (c_i (t^2-1)^k)\le 2k-2$. Thus $c_i = 0$ and, from the constant term, $H_{i,(0,0,3k)}=H_i(0,0,1) = 0$ for all $i=1,\dots,r$. This contradicts \eqref{E:2}.
\end{proof}

In the case of even $k$, the monomials $X_1^{\alpha_1}X_2^{\alpha_2}X_3^{\alpha_3}$ entering $H_i$ 
with $\alpha_1,\alpha_2$ even include $X_3^{3k}$, $X_1^{2k}X_2^{k}$ and $X_1^{k}X_2^{2k}$, so
$H_i(1,t,1)$ has degree $2k$, which does not prevent $H_i(1,t,1)$ from being a multiple of $(t^2-1)^k$.

\begin{remark}
Using similar arguments, one can show stubbornness of the form $S\in \cE(P_{3,6})\cap \Delta_{3,6}$ from \eqref{eq:S,Q}.
\end{remark}

\subsection{Stengle's form}\label{sub:Stengle}

Stengle showed \cite{Sten} that the ternary sextic
\begin{align*}
 T\ =\ X_1^3X_3^3 + (X_2^2X_3 - X_1^3 - X_1X_3^2)^2     
\end{align*}
is stubborn. It is interesting to note that $T$ is not extremal, as we now explain. It is easy to check that $T$ has $2$ real zeros at $[0:0:1]$ and $[0:1:0]$ (see, for example, the proof of Proposition \ref{prop:Stengle} below).
By \eqref{eq:Stengle_delta} we have that $\delta_{[0:0:1]}(T)=3$.
In the next example we compute the delta invariant of $T$ at $[0:1:0]$ showing that $\delta_{[0:1:0]}(T)=6$.

\begin{example}\label{ex:Stengle2}
  The strict transforms of
  \begin{align}\label{eq:t}
    t\ =\ (x_3-x_1^3-x_1x_3^2)^2+x_1^3x_3^3
  \end{align}
under two consecutive blow-ups $x_3=x_1x_3'$ and $x_3'=x_1x_3''$ are given as
  \begin{align*}
    t'\ &=\ (x_3'-x_1^2-x_1^2x_3'^2)^2+x_1^4x_3'^3,\quad t''\ =\ (x_3''-x_1-x_1^3x_3''^2)^2+x_1^5x_3''^3
  \end{align*}
  and hence $\delta_{(0,0)}(t)=1+\delta_{(0,0)}(t')=1+1+\delta_{(0,0)}(t'')$.
  Blowing up with $x_3''=x_1x_3'''$ gives
  \begin{align}\label{eq:t'''}
    t'''\ &=\ (x_3'''-1-x_1^4x_3'''^2)^2+x_1^6x_3'''^3
  \end{align} 
  Since the first order infinitely near point of $t''=0$ at $(x_1,x_3'')=(0,0)$ is at $(x_1,x_3''')=(0,1)$, it is convenient to write $t'''$ in the coordinates $x_1=\tilde x_1$, $x_3'''=\tilde x_3+1$,
  \begin{align*}
    \tilde t\ &=\ (\tilde x_3-\tilde x_1^4(\tilde x_3+1)^2)^2+\tilde x_1^6(\tilde x_3+1)^3,
  \end{align*}
so that $\delta_{(0,1)}(t''')=\delta_{(0,0)}(\tilde t)$.
Now, the consecutive blow-ups $\tilde x_3=x_1\tilde x_3'$ and $\tilde x_3'=x_1\tilde x_3''$ give
\begin{align*}
  \tilde t'\ &=\ (\tilde x_3'-\tilde x_1^3(\tilde x_1\tilde x_3'+1)^2)^2+\tilde x_1^4(\tilde x_1\tilde x_3'+1)^3,\quad
  \tilde t''\ =\ (\tilde x_3''-\tilde x_1^2(\tilde x^2_1\tilde x_3''+1)^2)^2+\tilde x_1^2(\tilde x^2_1\tilde x_3''+1)^3
\end{align*}
Finally, $\tilde t''$ has an ordinary singularity at $(0,0)$ and $\delta_{(0,0)}(\tilde t)=1+\delta_{(0,0)}(\tilde t')=1+1+\delta_{(0,0)}(\tilde t'')=3$.
Summarizing, we obtain that $\delta_{(0,0)}(t)=2+\delta_{(0,0)}(t'')=3+\delta_{(0,1)}(t''')=3+\delta_{(0,0)}(\tilde t)=3+3=6$.
\end{example}

As a consequence, the total delta invariant (equivalently, the SOS-invariant) of $T$ is equal to $\delta(T)=\delta^{\,\textrm{sos}}(T)=3+6=9$. As it is less than $10$, by Remark \ref{rem:count_proved},  Stengle's form $T\in \partial P_{3,6}$ is not extremal.
Proposition \ref{prop:Stengle} below gives an explicit representation of $T$ as a convex combination of two nonnegative sextics that are not proportional to it.

\begin{proposition}\label{prop:Stengle}
The form $T\in \partial P_{3,6}$ is not extremal, that is, $T\notin \cE(P_{3,6})$.
\end{proposition}
\begin{proof}
Let $c>0$ be a positive real number and consider 
\begin{align*}
T_c\ =\ c X_1^3X_3^3 + (X_2^2X_3 - X_1^3 - X_1X_3^2)^2
\end{align*}
so that $T = T_1$. Stengle's argument \cite{Sten} shows that $T_c \notin \Sigma_{3,6}(\infty)$, whether or not it is nonnegative.

We have $T_c(X_1,X_2,0) = X_1^6 \ge 0$, so to check whether $T_c$ is nonnegative, it is enough to consider the dehomogenized polynomial
\begin{align*}
T_c(X_1,X_2,1)\ =\ c X_1^3 + (X_2^2 - X_1^3 - X_1)^2.
\end{align*}
Note that $T_c(X_1,X_2,1) \geq 0$ when $X_1 \geq 0$ and if $X_1 < 0$, then $-X_1^3-X_1 > 0$, so
$T_c(X_1,X_2,1) \geq T_c(X_1,0,1)$. Thus, $T_c(X_1,X_2,1) \geq 0$ for all $(X_1,X_2)\in \mathbb{R}^2$ if and only if
\begin{align*}
T_c(X,0,1)\ =\ c X^3 + (-X^3 - X)^2 = X^2(cX + (X^2 + 1)^2)\ \geq\ 0,\quad X\in \mathbb{R}.
\end{align*}
An easy calculation shows that the  largest value of $c>0$ for which it is possible is $c = \sqrt{256/27} \approx 3.079$ and 
\begin{align}\label{eq:256/27}
T_{\sqrt{256/27}}(X,0,1) = X^2 \left( X + \frac{1}{\sqrt{3}}\right)^2\left(X^2 - \frac{2}{\sqrt{3}} X+ 3\right)\ \geq\ 0.
\end{align}
Thus,  $T_{\sqrt{256/27}}$ is nonnegative and $T=T_1$ is a convex combination of $T_{\sqrt{256/27}}$ and $T_0$.
\end{proof}

\begin{remark}
  By the same computation as in Examples \ref{ex:Stengle} and \ref{ex:Stengle2}, we can show that $\delta_{[0:0:1]}(T_c)=3$ and $\delta_{[0:1:0]}(T_c)=6$ for any $c\neq 0$.
\end{remark}
It follows from the above proof that $T$ belongs to the relative interior of a face $\mathcal{F}\subset \partial P_{3,6}$ of the cone of nonnegative ternary sextics with $\dim \mathcal{F}\geq 2$.
We now show that the dimension of this face $\mathcal{F}$ is exactly two.

\begin{proposition}\label{prop:fact_T}
The unique smallest face $\mathcal{F}$ of $P_{3,6}$ containing $T$ in its relative interior is $2$-dimensional. It is generated by $T_0$ and $T_{\sqrt{256/27}}$ both of which are extremal in $P_{3,6}$. 
\end{proposition}

\begin{proof}
 The idea of the proof is to show that a form $F\in \partial P_{3,6}$ contained in the relative interior of the face $\mathcal{F}\subset P_{3,6}$, must satisfy $26$ linear conditions, which define the plane spanned by $T_0$ and $T_{\sqrt{256/27}}$. 
  If $F=\sum_{\vert\alpha\vert=6} F_{\alpha}\X^{\alpha}\in \mathcal{F}$ is such a nonnegative sextic (written in the monomial basis), then we can write $T=F+\tilde F$ for some other nonnegative $\tilde F\in \mathcal{F}$.
 By \cite[Thm. $1$]{R} we have that $\New(F)$ is contained in $\New(T)=\textrm{conv}((6,0,0),(2,0,4),(0,4,2))$. Thus, we can write (see Figure \ref{fig:A5})
  \begin{align*}
    F\ &=\ F_{(2,0,4)}X_1^2X_3^4+F_{(1,2,3)}X_1X_2^2X_3^3+F_{(0,4,2)}X_2^4X_3^2+F_{(3,0,3)}X_1^3X_3^3+F_{(2,1,3)}X_1^2X_2X_3^3\\ &+F_{(4,0,2)}X_1^4X_3^2+F_{(3,1,2)}X_1^3X_2X_3^2+F_{(2,2,2)}X_1^2X_2^2X_3^2+F_{(1,3,2)}X_1X_2^3X_3^2+F_{(5,0,1)}X_1^5X_3\\ &+F_{(4,1,1)}X_1^4X_2X_3+F_{(3,2,1)}X_1^3X_2^2X_3+F_{(6,0,0)}X_1^6.
  \end{align*}
  Let us first treat $[0:0:1]$ and consider dehomogenizations $f_1(x_1,x_2)=F(x_1,x_2,1)$, $\tilde f_1(x_1,x_2)=\tilde F(x_1,x_2,1)$, $t_1(x_1,x_2)=T(x_1,x_2,1)$.
  The strict transform of $f_1$ (in the coordinates $(x_1',x_2)$ with $x_1=x_1'x_2$) is
 {\small\begin{align*}
   f_1'\ &=\ F_{(2,0,4)}x_1'^2+F_{(1,2,3)}x_1'x_2+F_{(0,4,2)}x_2^2 +F_{(3,0,3)}x_1'^3x_2+F_{(2,1,3)}x_1'^2x_2+F_{(4,0,2)}x_1'^4x_2^2+F_{(3,1,2)}x_1'^3x_2^2\\ &+F_{(2,2,2)}x_1'^2x_2^2+F_{(1,3,2)}x_1'x_2^2+F_{(5,0,1)}x_1'^5x_2^3+F_{(4,1,1)}x_1'^4x_2^3+F_{(3,2,1)}x_1'^3x_2^3+F_{(6,0,0)}x_1'^6x_2^4                                  
 \end{align*}}
The second blow-up at $(x_1',x_2)=(0,0)$ (in the coordinates $(x_1'',x_2)$, where $x_1'=x_1''x_2$) reveals
{\small\begin{align*}
  f_1''\ &=\ F_{(2,0,4)}x_1''^2+F_{(1,2,3)}x_1''+F_{(0,4,2)} +F_{(3,0,3)}x_1''^3x^2_2+F_{(2,1,3)}x_1''^2x_2+F_{(4,0,2)}x_1''^4x_2^4+F_{(3,1,2)}x_1''^3x_2^3\\ &+F_{(2,2,2)}x_1''^2x_2^2+F_{(1,3,2)}x_1''x_2 +F_{(5,0,1)}x_1''^5x_2^6 +F_{(4,1,1)}x_1''^4x_2^5+F_{(3,2,1)}x_1''^3x_2^4+F_{(6,0,0)}x_1''^6x_2^8
\end{align*}}We also have $t''_1=f''_1+\tilde f_1''$. By Lemma \ref{lem:nonnegative_blow} and Example \ref{ex:Stengle}, $f_1''$ must have a singular point at $(x_1'',x_2)=(1,0)$, which gives $F_{(2,0,4)}+F_{(1,2,3)}+F_{(0,4,2)}=0$, $2 F_{(2,0,4)}+F_{(1,2,3)}=0$ and $F_{(2,1,3)}+F_{(1,3,2)}=0$ or, equivalently,
\begin{align}\label{eq:localFFF}
  F_{(2,0,4)}\ =\ F_{(0,4,2)},\quad F_{(1,2,3)}\ =\ -2 F_{(0,4,2)},\quad F_{(1,3,2)}\ =\ -F_{(2,1,3)}.
\end{align}

We now turn to the point $[0:1:0]$ and consider $f_2(x_1,x_3)=F(x_1,1,x_3)$, $\tilde f_2(x_1,x_3)=\tilde F(x_1,1,x_3)$, $t_2(x_1,x_3)=T(x_1,1,x_3)$.
Following a sequence of blow-ups $x_3=x_1x_3'$, $x_3'=x_1x_3''$ and  $x_3''=x_1x_3'''$  from Example \ref{ex:Stengle2} we obtain

     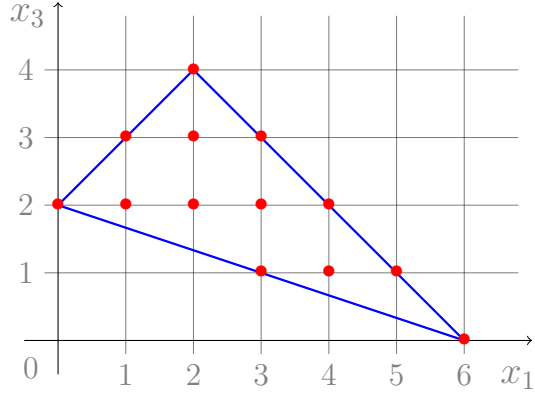
\begin{figure}[h]
    \centering
  \begin{tikzpicture}[xscale=0.09,yscale=0.09,domain=0.125:220,samples=400]
    \node[opacity=0.5] at (-4,-4) {$0$};

    \draw[->] (-5,0) -- (70,0) node[below] {};
    \draw[->] (0,-5) -- (0,50) node[left] {};
    \foreach \i in {1,2,...,6} {
        \draw[opacity=0.5] (\i*10,48) -- (\i*10,-2) node[below] {$\i$};
    }
    \foreach \i in {1,2,...,4} {
        \draw[opacity=0.5] (68,\i*10) -- (-2,\i*10) node[left] {$\i$};
    }
    \draw[line width = 0.3mm, blue] (60,0) -- (20,40);
    \draw[line width = 0.3mm, blue] (60,0) -- (0,20);
    \draw[line width = 0.3mm, blue] (0,20) -- (20,40);



  \node[opacity=0.5] at (68,-5.5) {{\large $x_1$}};
  \node[opacity=0.5] at (-4.5,48) {{\large $x_3$}};
   
    \draw[red] (0,20) node {$\bullet$};
    \draw[red] (10,20) node {$\bullet$};
    
    \draw[red] (60,0) node {$\bullet$};
    \draw[red] (30,10) node {$\bullet$};
    \draw[red] (40,10) node {$\bullet$};
    \draw[red] (50,10) node {$\bullet$};
    \draw[red] (20,20) node {$\bullet$};
    \draw[red] (30,20) node {$\bullet$};
    \draw[red] (40,20) node {$\bullet$};
    \draw[red] (10,30) node {$\bullet$};
    \draw[red] (20,30) node {$\bullet$};
    \draw[red] (30,30) node {$\bullet$};
    \draw[red] (20,40) node {$\bullet$};
  \end{tikzpicture}   
        \caption{\centering The maximal support and the Newton polytope of a dehomogenization $f_2(x_1,x_3)=F(x_1,1,x_3)$ of a sextic $F\in \mathcal{F}$.}  
    \label{fig:A5}
  \end{figure}


{\small \begin{align*}
  f_2'\,&=\,F_{(0,4,2)}x_3'^2+F_{(3,2,1)}x_1^2x_3'+F_{(6,0,0)}x_1^4-F_{(2,1,3)}x_1x_3'^2-2F_{(0,4,2)}x_1^2x_3'^3+F_{(2,2,2)}x_1^2x_3'^2+F_{(2,1,3)}x_1^3x_3'^3\\ &+F_{(3,1,2)}x_1^3x_3'^2+F_{(4,1,1)}x_1^3x_3'
        +F_{(0,4,2)}x_1^4x_3'^4+F_{(3,0,3)}x_1^4x_3'^3+F_{(4,0,2)}x_1^4x_3'^2+F_{(5,0,1)}x_1^4x'_3,\\
  f_2''\,&=\, F_{(0,4,2)}x_3''^2+F_{(3,2,1)}x_1x_3''+F_{(6,0,0)}x_1^2-F_{(2,1,3)}x_1x_3''^2-2F_{(0,4,2)}x_1^3x_3''^3+F_{(2,2,2)}x_1^2x_3''^2\\ &+F_{(2,1,3)}x_1^4x_3''^3+F_{(3,1,2)}x_1^3x_3''^2+F_{(4,1,1)}x_1^2x_3''+F_{(0,4,2)}x_1^6x_3''^4 +F_{(3,0,3)}x_1^5x_3''^3+F_{(4,0,2)}x_1^4x_3''^2\\ &+F_{(5,0,1)}x_1^3x''_3,\\
  f_2'''\,&=\, F_{(0,4,2)}x_3'''^2+F_{(3,2,1)}x_3'''+F_{(6,0,0)}-F_{(2,1,3)}x_1x_3'''^2-2F_{(0,4,2)}x_1^4x_3'''^3+F_{(2,2,2)}x_1^2x_3'''^2\\
          &+F_{(2,1,3)}x_1^5x_3'''^3+F_{(3,1,2)}x_1^3x_3'''^2+F_{(4,1,1)}x_1x_3'''+F_{(0,4,2)}x_1^8x_3'''^4+F_{(3,0,3)}x_1^6x_3'''^3+F_{(4,0,2)}x_1^4x_3'''^2\\ & +F_{(5,0,1)}x_1^2x'''_3.
\end{align*}}Since $t_2'''=f_2'''+\tilde f_2'''$, Lemma \ref{lem:nonnegative_blow} and Example \ref{ex:Stengle2} imply that $f_2'''$ must have a singular point at $(x_1,x_3''')=(0,1)$, that is,
$F_{(0,4,2)}+F_{(3,2,1)}+F_{(6,0,0)}=0$, $-F_{(2,1,3)}+F_{(4,1,1)}=0$ and $2F_{(0,4,2)}+F_{(3,2,1)}=0$ or, equivalently, 
\begin{align}\label{eq:3cond}
  F_{(6,0,0)}\ =\ F_{(0,4,2)},\quad F_{(3,2,1)}\ =\ -2F_{(0,4,2)},\quad F_{(4,1,1)}\ =\ F_{(2,1,3)}.
\end{align}
Looking at the leading terms of \eqref{eq:t'''} at $(0,1)$ (cf. Figure \ref{fig:A5}), we require vanishing of
{\begin{align*}
  \frac{\partial^2 f_2'''}{\partial x_1^2}(0,1)\ &=\ 2(F_{(2,2,2)}+F_{(5,0,1)}),\qquad
  \frac{\partial^2 f_2'''}{\partial x_1 \partial x_3'''}(0,1)\ =\ -2F_{(2,1,3)}+F_{(4,1,1)},\\  
  \frac{\partial^3 f_2'''}{\partial x_1^3}(0,1)\ &=\ 6F_{(3,1,2)},\qquad \frac{\partial^3 f_2'''}{\partial x_1^2 \partial x_3'''}(0,1)\ =\ 4 F_{(2,2,2)}+2F_{(5,0,1)},\\ \frac{\partial^4 f_2'''}{\partial x_1^4}(0,1)\ &=\ 4!(-2F_{(0,4,2)}+F_{(4,0,2)}),\qquad \frac{\partial^5 f_2'''}{\partial x_1^5}(0,1)\ =\ 5!F_{(2,1,3)}
\end{align*}}
Combining these conditions with  \eqref{eq:3cond} we update $F$:
\begin{align*}
 F\ &=\ F_{(0,4,2)}(X_1^2X_3^4-2X_1X_2^2X_3^3+X_2^4X_3^2+2X_1^4X_3^2-2X_1^3X_2^2X_3+X_1^6)+F_{(3,0,3)}X_1^3X_3^3\\ &=\ F_{(0,4,2)}(X_2^2X_3-X_1^3-X_1X_3^2)^2+F_{(3,0,3)}X_1^3X_3^3.
\end{align*}
Since $F\in \mathcal{F}\subset \partial P_{3,6}$ is nonnegative, $F_{(0,4,2)}$ cannot be zero (and hence has to be positive).
On the other hand, $F_{(3,0,3)}=F(1,\sqrt{2},1)$ must be nonnegative. It follows that $F$ is proportional to $T_c$ for some $c\geq 0$ and so the face $\mathcal{F}$ is two-dimensional.

It is easy to argue that $T_0$, which is a single square of an indefinite cubic, is an extreme ray of $P_{3,6}$. The form $T_{\sqrt{256/27}}$ has zeroes at $[0:0:1]$ and $[0:1:0]$, with $\delta_{[0:0:1]}(T_{\sqrt{256/27}})=3$ and $\delta_{[0:1:0]}(T_{\sqrt{256/27}})=6$. Additionally, the form $T_{\sqrt{256/27}}$ has also a round zero at $[1:0:-\sqrt{3}]$ (cf. \eqref{eq:256/27}) which makes
  \begin{align*}
    \delta(T_{\sqrt{256/27}})\ =\ \delta_{[0:0:1]}(T_{\sqrt{256/27}})+\delta_{[0:1:0]}(T_{\sqrt{256/27}})+\delta_{[1:0:-\sqrt{3}]}(T_{\sqrt{256/27}})\ =\ 3+6+1\ =\ 10.
  \end{align*}
Therefore, $T_{\sqrt{256/27}}\in \cE(P_{3,6})$ is extremal, and the face $\mathcal{F}$ is generated by two extreme rays $T_0$ and $T_{\sqrt{256/27}}$.
\end{proof}






\section{$n$-variate forms}\label{sec:higher_n}


We now discuss $n$-variate forms ($n>3$) with the perspective of understanding the relation between $\Sigma_{n,d}(\infty)$ and $P_{n,d}$. We first show that the quaternary quartic \eqref{eq:S,Q} is stubborn.

\begin{theorem}\label{thm:Q}
The quaternary quartic $Q$ from \eqref{eq:S,Q} is stubborn.
\end{theorem}

\begin{proof}
If we dehomogenize $Q = X_4^4 + X_1^2X_2^2 + X_1^2X_3^2 + X_2^2X_3^2 -4X_1X_2X_3X_4$ by setting $X_4 = 1$, and simultaneously set $X_3 = x_1x_2$, then  (c.f. \cite[Prop. 3.7]{CL})
\begin{align*}
  Q(x_1,x_2,x_1x_2,1) &=\ 1 + x_1^2x_2^2 + x_1^4x_2^2 + x_1^2x_2^4 - 4x_1^2x_2^2\ =\ M(x_1,x_2,1).
\end{align*}
Thus, if $Q^k = \sum_{i=1}^r H_i^2\in \Sigma_{4,4k}$ was a sum of squares for some odd $k\geq 1$, then so would be
\begin{align*}
M^k(x_1,x_2,1)\ =\ Q^k(x_1,x_2,x_1x_2,1)\ =\ \sum_{i=1}^r H_i(x_1,x_2,x_1x_2,1)^2,
\end{align*}
which we know to be impossible by Theorem \ref{thm:main} and Remark \ref{rem:hom}. 
\end{proof}



\smallskip

As a consequence of Theorem \ref{thm:main}, results from \cite{Brugall2018RealAC} and Theorem \ref{thm:Q} we now obtain that $\Sigma_{n,d}(\infty)=P_{n,d}$ if and only if $\Sigma_{n,d}=P_{n,d}$, see Figure \ref{fig:table}.

\begin{theorem}\label{thm:general}
Let $\Delta_{n,d}\neq \emptyset$, that is, $n\geq 3$ and $d\geq 6$ or $n\geq 4$ and $d\geq 4$. Then there exists a stubborn form $P\in P_{n,d}$.\end{theorem}

\begin{proof}
 By Proposition \ref{prop:ext} or by combining Theorems \ref{thm:main} and \ref{thm:higher_d} there are stubborn forms $P\in \Delta_{3,d}$ for any even $d\geq 6$.
  By regarding such forms as forms in $P\in P_{n,d}\supset P_{3,d}$ in $n\geq 4$ variables, we see that $P^k$ cannot be a sum of squares whatever odd $k\geq 1$ one takes.
  Indeed, if $P^k$ was in $\Sigma_{n,d}$, then by setting $X_4=\dots=X_n=0$ we would have that $P^k=P^k|_{X_4=\dots=X_n=0}$ is in $\Sigma_{3,d}$, which is a contradiction.

  Similarly, considering the quaternary quartic $Q\in \Delta_{4,4}$ from \eqref{eq:S,Q}, which by Theorem \ref{thm:Q} satisfies $Q\notin \Sigma_{4,4}(\infty)$, and regarding it as a form $Q\in P_{n,4}\supset P_{4,4}$ in $n\geq 5$ variables, we obtain that $Q^k$ is not a sum of squares for any odd $k\geq 1$. 
\end{proof}

\begin{figure}[h]\centering
{\large\begin{tabular}{ |c|c|c|c|c| } 
 \hline
 $n\ \backslash\ d$ & 2 & 4 & 6 & $\geq 8$ \\
 \hline
 2  &  $\checkmark$ & $\checkmark$ & $\checkmark$ & $\checkmark$  \\ \hline
 3  &  $\checkmark$ & $\checkmark$ &  $\times$ & $\times$  \\
   \hline
  4   & $\checkmark$ & $\times$ & $\times$ & $\times$ \\
  \hline
  $\geq 5$ & $\checkmark$ & $\times$ & $\times$ & $\times$\\
 \hline
\end{tabular}}
\caption{Cases $(n,d)$ when $\Sigma_{n,d}(\infty)=P_{n,d}$ or, equivalently, when $\Sigma_{n,d}=P_{n,d}$.}\label{fig:table}
\end{figure}

Instead of trivially regarding a ternary form $P\in P_{3,d}\subset P_{n,d}$, $P\notin \Sigma_{3,d}(\infty)$, as a form in $n$ variables, one can consider an $n$-variate form $\tilde P(\X,\tilde \X):=P(\X+ A\tilde\X)$, where $\X=(X_1,X_2,X_3)$, $\tilde\X=(X_4,\dots, X_n)$ and $A$ being any $3\times (n-3)$ matrix. Then $\tilde P\notin \Sigma_{n,d}(\infty)$ by the same argument as in the proof of Theorem \ref{thm:general}.

An example of a stubborn nonnegative form in $5$ variables which cannot be obtained in the way described above is the Horn form $F = \bigg(\sum_{i=1}^5 X_i^2\bigg)^2 - 4\ \sum_{i=1}^5 X_i^2X_{i+1}^2$. The proof of the following result relies on the idea from \cite[p.$\,25$]{Diananda1962OnNF} (cf. \cite[Sect. $4$]{Powers2019ANO}).

%

\begin{theorem}\label{thm:Horn}
  The Horn form $F = \left(\sum_{i=1}^5 X_i^2\right)^2 - 4\ \sum_{i=1}^5 X_i^2X_{i+1}^2$ is stubborn.
\end{theorem}

\begin{proof}
Viewing the subscripts cyclically modulo 5, one sees that the coefficient of $X_i^2X_j^2$ (for $i\neq j$) in $F$ is $-2$ 
(resp. 2) if $\vert i-j\vert = 1$ (resp. $\vert i-j\vert = 2$) and 
\begin{align*}
F(X_1,X_2,X_3,X_4,X_5)\ &=\ F(X_2,X_3,X_4,X_5,X_1)\ =\ F(X_3,X_4,X_5,X_1,X_2)\\ &=\ F(X_4,X_5,X_1,X_2,X_3)\ =\ F(X_5,X_1,X_2,X_3,X_4), 
\end{align*}
that is, $F$ is cyclically symmetric.
We have an alternative representation
\begin{align}\label{eq:alt_rep}
F(X_1,\dots,X_5) = (X_1^2-X_2^2+X_3^2-X_4^2+X_5^2)^2 + 4(X_2^2-X_1^2)X_5^2 + 4X_1^2X_4^2, 
\end{align}
from which it is straightforward to see that $F\in P_{5,4}$ is nonnegative (one can assume that $X_1^2\leq X_2^2$ by cyclic symmetry).

Suppose now $F = \sum_{i=1}^r H_i^2\in \Sigma_{5,4}$ is a sum of squares and let the coefficient of $X_{j}^2$ in $H_i$ be denoted by $H_{ij}$. Then
\begin{align*}
(X_1^2-X_2^2+X_3^2)^2\ =\ F(X_1,X_2,X_3,0,0)\ =\ \sum_{i=1}^r H_i^2(X_1,X_2,X_3,0,0).
\end{align*}
Since the quadratic form $H_i(X_1,X_2,X_3,0,0)$ vanishes on 
$X_1^2-X_2^2+X_3^2=0$, it must be a multiple of $X_1^2-X_2^2+X_3^2$ and thus,  $H_{i1} = -H_{i2} = H_{i3}$. 
By cycling the variables, we see that $H_{i2} = -H_{i3} = H_{i4}$ as well as $H_{i3} = -H_{i4} = H_{i5}$ and $H_{i4} = -H_{i5} = H_{i1}$,
so that $H_{i1} = -H_{i1} = 0$ for all $i=1,\dots,r $. This implies that the coefficient of $X_1^4$ in $F=\sum_{i=1}^r H^2_i$ is 
$\sum_{i=1}^r H_{i1}^2 = 0$, which is a contradiction.

The proof is nearly identical if we take an odd power of $F$. Suppose 
$F^{k} = \sum_{i=1}^r H_i^2\in \Sigma_{5,4k}$ is a sum of squares and let $H_{ij}$ be the coefficient of $X_{j}^{2k}$ in $H_i$.  
Then
\begin{align}\label{eq:F^k}
(X_1^2-X_2^2+X_3^2)^{2k} = F^{k}(X_1,X_2,X_3,0,0)\ =\ \sum_{i=1}^r H_i^2(X_1,X_2,X_3,0,0).
\end{align}

We show by induction that if
\begin{align}\label{eq:ind}
  (X_1^2-X_2^2+X_3^2)^{2k}\ =\ \sum_{i=1}^r Q^2_i(X_1,X_2,X_3),
\end{align}
then $Q_i(X_1,X_2,X_3) = \al_i (X_1^2-X_2^2+X_3^2)^{k}$ for some $\alpha_i\in \R$.
If $k=1$, the claim follows by the above argument.
Otherwise, since $Q_i$  vanish on $X_1^2-X_2^2+X_3^2=0$, it must be divisble by $X_1^2-X_2^2+X_3^2$ and we can write $Q_i= (X_1^2-X_2^2+X_3^2)\tilde Q_i$ for some degree $2k-2$ forms $\tilde Q_i\in \R[X_1,X_2,X_3]$. We factor out $(X_1^2-X_2^2+X_3^2)^2$ from \eqref{eq:ind} and apply the induction hypothesis to derive the claim. 

Applying this result to our alleged sum of squares \eqref{eq:F^k}, we see that each $H_i(X_1,X_2,X_3,0,0)$ is a multiple of 
$(X_1^2-X_2^2+X_3^2)^{k}$,
and so by looking at $X_{j}^{2k}$ in $H_i$, we see that  $H_{i1} = -H_{i2} = H_{i3}$. 
Again, by cycling the variables,  we obtain that $H_{i2} = -H_{i3} = H_{i4}$, $H_{i3} = -H_{i4} = H_{i5}$ and $H_{i4} = -H_{i5} = H_{i1}$,
so that $H_{i1} = -H_{i1} = 0$ for all $i=1,\dots, r$. From this we have that the coefficient of $X_1^{4k}$ in $F^k=\sum_{i=1}^r H^2_i$ is 
$\sum_{i=1}^r H_{i1}^2 = 0$, which is a contradiction.

\end{proof}

In contrast to $Q\in P_{4,4}\setminus \Sigma_{4,4}(\infty)$ and forms in $P_{3,d}\setminus \Sigma_{3,d}(\infty)$ considered in Section \ref{sec:ternary}, the Horn form $F\in P_{5,4}\setminus \Sigma_{5,4}(\infty)$ has infinitely many real zeros. Indeed, by the proof of Theorem \ref{thm:Horn} (see \eqref{eq:alt_rep}), the Horn form vanishes on the plane curve $\cV(X_1^2-X_2^2+X_3^2, X_4, X_5)$ as well as on the curves $\cV(X_2^2-X_3^2+X_4^2, X_5, X_1),\dots, \cV(X_5^2-X_1^2+X_2^2, X_3, X_4)$ obtained from the former one by cyclically permuting the variables.

\section{Convexity of $\Sigma_{n,d}(\infty)$}\label{sec:convexity}
Our goal here is to prove Theorem \ref{3.3} below, which implies that the set $\Sigma_{n,d}(\infty)$ of non-stubborn forms is convex.
We begin with a special case. 
\begin{theorem}\label{3.1}
Suppose $P_1$ and $P_2$ are nonnegative forms such that $P_1^k$ and $P_2$ are both sums of squares for some odd $k\geq 1$. Then $(P_1+P_2)^{k}$ is also a sum of squares. 
\end{theorem}

\begin{proof}
The function $x^k$, $x\in \R$, is strictly monotone. Then for any real $t_1,t_2$,
\begin{align*}
F_{k}(t_1,t_2)\ :=\  \frac{(t_1+t_2)^{k} - t_1^{k}}{(t_1+t_2) - t_1} = \sum_{i=0}^{k-1}{k \choose i+1}\,  t_1^{k-1-i}t_2^i 
\end{align*}
is a quotient of two positive numbers when $t_2 > 0$ and two negative numbers when $t_2 < 0$, which implies that $F_{k}\in P_{2,k-1}$ considered as a binary form is nonnegative.
Thus, $F_{k}(t_1,t_2) = G_k^2(t_1,t_2) + H_k^2(t_1,t_2)$  for some forms $G_k,H_k$ of degree $(k-1)/2$.
The claim now follows, since 
\begin{align*}
(P_1 + P_2)^{k}\ &=\ P_1^{k} + P_2 \cdot \frac{(P_1 + P_2)^{k} - P_1^{k}}{(P_1+P_2) - P_1}\ =\ P_1^{k} + P_2 \cdot F_{k}(P_1,P_2) \\
&=\  P_1^{k} + P_2 \cdot \left(G_k^2(P_1,P_2) + H_k^2(P_1,P_2)\right).
\end{align*}
is a sum of squares, since it is a sum of a sum of squares form and a product of sum of squares forms.
\end{proof}

For the more general case, we need to look at other truncated binomials sums. 
For integers $n\geq r\geq 0$ let
\begin{align*}
  f_{n,r}(t)\ =\ \sum_{i=0}^r \binom ni\, t^i
\end{align*}
denote the truncated binomial polynomial. Extensive numerical computation suggested that if 
$n > 2r$, the polynomial $f_{n,2r}>0$ is positive. The third author
posed this as an unproved conjecture on MathOverflow, and the solution was quickly presented by Prof. Iosif Pinelis, which is included with his kind permission.
\begin{theorem}[Pinelis,  \cite{MO}]\label{MO}
For all $n > 2r$ and $t \in \rr$, $f_{n,r}(t)> 0$.
\end{theorem}

\begin{proof}
We argue by induction on $n$ for fixed $r$. As before, $f_{2r+1,2r}(t) = (1+t)^{2r+1} - t^{2r+1}$ is positive
because $x^{2r+1}$ is strictly increasing. 
We use two combinatorial identities:
\begin{align}
f_{n,r}'(t)\ &=\  \sum_{i=0}^r \binom ni\  i t^{i-1}\ =\ n \sum_{i=1}^r \binom {n-1}{i-1}\  t^{i-1}\ =\ nf_{n-1,r-1}(t),\label{eq1}\\
f_{n,r}(t)\ &=\ \sum_{i=0}^r \binom {n-1}i\, t^i + \sum_{i=1}^r \binom {n-1}{i-1}\, t^i\ =\ f_{n-1,r}(t) + tf_{n-1,r-1}(t).\label{eq2}
\end{align}

Since $f_{n,r}(t) \to \infty$ as $t \to \pm \infty$, it suffices to look at values of $f_{n,r}$ at its critical points. Suppose $f_{n,r}'(t_0) = 0$. Then
by \eqref{eq1}, $f_{n-1,r-1}(t_0) = 0$ and so by \eqref{eq2},
\[
f_{n,r}(t_0)\ =\  f_{n-1,r}(t_0) + t_0f_{n-1,r-1}(t_0)\ =\ f_{n-1,r}(t_0),
\]
which is positive by the inductive hypothesis. 
\end{proof}

Observe that by this theorem there are polynomials $g_{n,2r}(t)$ and $h_{n,2r}(t)$ 
such that $f_{n,2r}(t) = (g_{n,2r}(t))^2 +  (h_{n,2r}(t))^2$, and upon homogenization that there exist binary forms $G, H$
of degree $r$ so that
\[
F_{n,2r}(t_1,t_2)\ =\ \sum_{i=0}^r \binom ni\, t_1^it_2^{2r-i}\ =\  (G_{n,2r}(t_1,t_2))^2 +  (H_{n,2r}(t_1,t_2))^2.
\]
We need this representation in the main result of this section. 
\begin{theorem}\label{3.3}
  Let $P\in \Sigma_{n,d}(k)$ and $\tilde P\in \Sigma_{n,d}(\tilde k)$, where numbers $k$ and $\tilde k$ are  both odd.
  Then $P+\tilde P\in \Sigma_{n,d}(k+\tilde k-1)$.
  In particular, $\Sigma_{n,d}(\infty)$ is a convex cone.
\end{theorem}

\begin{proof}
We expand $(P+\tilde P)^{k+\tilde k-1}$, and using $i' = i - k$ below, we obtain its representation as a sum of squares:
\begin{align*}
(P+\tilde P)^{k+\tilde k-1} &= \sum_{i = 0}^{k+\tilde k-1} \binom{k+\tilde k -1}i P^{i}\tilde P^{\,k+\tilde k-1-i}\\ 
&=\ \sum_{i = 0}^{k-1} \binom{k+\tilde k-1}i P^{i}\tilde P^{k+\tilde k-1-i}  +  \sum_{i = k}^{k+\tilde k-1} \binom{k+\tilde k-1}i P^{i}\tilde P^{k+\tilde k-1-i}\\
&=\ \tilde P^{\tilde k} \sum_{i = 0}^{k-1} \binom{k+\tilde k-1}i P^{i}\tilde P^{k-1-i}  +
 P^{k}  \sum_{i' = 0}^{\tilde k-1} \binom{k+\tilde k-1}{i '+k}P^{i'}\tilde P^{\tilde k-1-i'}\\
&=\ \tilde P^{\tilde k} \sum_{i = 0}^{k-1} \binom{k+\tilde k-1}i P^{i}\tilde P^{k-1-i}  +
 P^{k}  \sum_{i' = 0}^{\tilde k-1} \binom{k+\tilde k-1}{\tilde k-1-i'}P^{i'}\tilde P^{\tilde k-1-i'}\\
 &=\ \tilde P^{\tilde k}F_{k+\tilde k-1,k-1}(P,\tilde P) + P^{k}F_{k+\tilde k-1,\tilde k-1}(\tilde P,P) \\
 &=\  \tilde P^{\tilde k}(G_{k+\tilde k-1,k-1}^2(P,\tilde P) +  H_{k+\tilde k-1,k-1}^2( P,\tilde  P)) \\
 &\ \ \ \,+ P^{k}(G^2_{k+\tilde k-1,\tilde k-1}(\tilde P,P) +  H_{k+\tilde k-1,\tilde k-1}^2(\tilde P,P)).
\end{align*}

The second claim follows immediately.

\end{proof}


\section{Further remarks and questions}

In this section we make some remarks around our results and state some open questions.

Conjecture \ref{conj:general} asserts that any extremal form in $P_{n,d}$ that is not a sum of squares is stubborn.

Also, we do not know whether for a fixed odd $k$, the closed cone $\Sigma_{n,d}(k)$ is convex or not.
We expect the answer should not depend on the values of the parameters.

\smallskip

We showed in Subsection \ref{sub:Stengle} that Stengle's form \eqref{eq:Stengle} is not extremal. In fact, by Proposition \ref{prop:fact_T} the face containing $T$ in its relative interior is two-dimensional.
One can show that all forms in the relative interior of this face are stubborn. This motivates the following question: ``What is the maximal dimension of a face $\mathcal{F}\subset P_{n,d}$ whose relative interior consists of stubborn forms?"



Extremal ternary sextics $P\in \cE(P_{3,6})\cap \Delta_{3,6}$ satisfy $\delta^{\,\textrm{sos}}(P)=\delta(P)=10$ and are stubborn by Theorem \ref{thm:main}.
We saw in Subsection \ref{sub:Stengle} that $\delta^{\,\textrm{sos}}(T)=\delta(T)=9$ for Stengle's form, which is stubborn as well \cite{Sten}. Based on this and Proposition \ref{prop:fact_T} we make the following conjecture.

\begin{conjecture}
Let $P\in \Delta_{3,6}$ be a nonnegative form that is not a sum of squares, such that $\delta(P)=9$. Then $P$ is stubborn and lies in the relative interior of a $2$-dimensional face of $P_{3,6}$.
\end{conjecture}
It is interesting to see whether there are stubborn nonnegative forms with a ``small'' SOS-invariant (equivalently, $\delta$-invariant).
For example, are there $P\in \Delta_{3,6}(\infty)$ with $\delta(P)<9$?

The Robinson form \eqref{eq:Robinson} (as well as any $P\in \mathcal{E}(P_{3,6})\cap \Delta_{3,6}$ that generates an exposed extreme ray) has $10$ round zeros. The Motzkin form \eqref{eq:Motzkin} has four round zeros $[\pm 1:\pm 1:1]$ and two zeros $[1:0:0]$, $[0:1:0]$ at each of which the SOS-invariant is equal to $3$. This motivates the following problem.

\begin{problem}
Classify partitions $(\delta_1,\dots, \delta_s)$ of $10=\delta_1+\dots+\delta_s$ that can be realized as the set of local SOS-invariants (equivalently, delta invariants) at real zeros of some $P\in \mathcal{E}(P_{3,6})\cap \Delta_{3,6}$.
\end{problem}

\smallskip

For $ a \le 3$, let us now define
\[
M_a\ =\ X_1^4X_2^2 + X_1^2X_2^4 + X_3^6 - aX_1^2X_2^2X_3^2\ =\ M + (3-a)X_1^2X_2^2X_3^2
\]
and consider for $k\geq 0$ the set
 \[
V_{2k+1}\ =\ \{a\,:\, M_a^{2k+1} \in \Sigma_{3,6(2k+1)}\}\ =\ \{ a \,:\, M_a \in \Sigma_{3,6}(2k+1)\}
\]
of parameters $a\leq 3$ for which the $(2k+1)$-th power of $M_a$ is a sum of squares. 
\begin{theorem} There exists a non-decreasing sequence $(c_k) \subset [0,3)$ so that 
\[
V_{2k+1} = (-\infty, c_k],\quad k\geq 0.
\]
\end{theorem}
\begin{proof}
For each $k\geq 0$ we obviously have $(-\infty,0] \subseteq V_{2k+1}$ and $3 \notin V_{2k+1}$ by Corollary \ref{cor:main}. Let $c_k = \sup(V_{2k+1})$. Since the cone $\Sigma_{3,6(2k+1)}$ is closed, it follows that $c_k \in V_{2k+1}$ and so  $c_k \in [0,3)$.
Suppose $a = c_k - c$ is any number smaller than $c_k$ (so
$c > 0$). Then 
\[
X_1^4X_2^2 + X_1^2X_2^4 + X_3^6 - aX_1^2X_2^2X_3^2\ =\ \left(X_1^4X_2^2 + X_1^2X_2^4 + X_3^6 - c_kX_1^2X_2^2X_3^2\right)+ c X_1^2X_2^2X_3^2,
\]
is in $\Sigma_{3,6}(2k+1)$ by Theorem \ref{3.3} applied to forms $P=X_1^4X_2^2 + X_1^2X_2^4 + X_3^6 - c_kX_1^2X_2^2X_3^2\in \Sigma_{3,6}(2k+1)$ and $\tilde P = c X_1^2X_2^2X_3^2\in \Sigma_{3,6}$. 
\end{proof}

From the proof that $M$ is not a sum of squares \cite{Motzkin} one infers that $V_1 = (-\infty, 0]$, that is, $c_0=0$.
Now, the third power of $M_a$ satisfies
\begin{align*}
M_a^3\ &=\ \frac{3}{2} \left(X_1^5X_2^4 - a X_1^3X_2^4X_3^2\right)^2 + \frac{3}{2} \left(X_1^4X_2^5 - a X_1^4X_2^3X_3^2\right)^2 +\frac{3}{2} \left(X_1^4X_2^2X_3^3 - a X_1^2X_2^2X_3^5\right)^2\\ &+ \frac{3}{2} \left(X_1^2X_2^4X_3^3 - a X_1^2X_2^2X_3^5\right)^2
+\frac{3}{2} \left(X_1X_2^2X_3^6 - a X_1^3X_2^4X_3^2\right)^2 + \frac{3}{2}\left(X_1^2X_2X_3^6 - a X_1^4X_2^3X_3^2\right)^2\\
&+\frac{3}{2} \left(X_1^2X_2^4X_3^3- X_1^4X_2^2X_3^3\right)^2 + \frac{3}{2}\left(X_1^4X_2^5 - X_1^2X_2X_3^6\right)^2 + \frac{3}{2}\left(X_1^5X_2^4 - X_1X_2^2X_3^6\right)^2\\ &
+ \left(X_3^9 - 2a X_1^2X_2^2X_3^5\right)^2 + a \left(X_1 X_2 X_3^7 - 2a X_1^3X_2^3X_3^3\right)^2
+ \left(X_1^3X_2^6 - 2a X_1^3X_2^4X_3^2\right)^2\\  &+ a \left(X_1^3X_2^5X_3 - 2a X_1^3X_2^3X_3^3\right)^2
+ \left(X_1^6X_2^3 - 2a X_1^4X_2^3X_3^2\right)^2 + a \left(X_1^5X_2^3X_3 - 2a X_1^3X_2^3X_3^3\right)^2\\
&+ \left(15 - 13a^3\right)\left(X_1^3X_2^3X_3^3\right)^2
\end{align*}
and as a consequence we have that $\frac{15}{13}^{1/3} \approx 1.04886\le c_1$.
Note that we cannot apply Scheiderer's theorem from \cite{Sch} here, because
$M_c$ for $c > 0$ is not strictly positive (it has real zeros at $[1:0:0]$ and $[0:1:0]$).
Pablo Parrilo has experimentally found out \cite{email} that $c_1 \approx 2.56548$ and also $c_2 \approx 2.88905$. One interesting question is whether $\lim_k c_k = 3$.

%





\end{document}